\documentclass[12pt]{amsart}
\usepackage{amsfonts,amsthm,amssymb,amsmath}
\usepackage{hyperref}

\def\Z{{\mathbb Z}}
\def\Qq{{\mathbb Q}}
\def\R{{\mathbb R}}
\def\C{{\mathbb C}}
\def\H{{\mathfrak H}}

\def\sq{\hbox{\rlap{$\sqcap$}$\sqcup$}}
\def\qed{\ifmmode\sq\else{\unskip\nobreak\hfil
         \penalty50\hskip1em\null\nobreak\hfil\sq
         \parfillskip=0pt\finalhyphendemerits=0\endgraf}\fi}

\def\smat#1#2#3#4{\left(\begin{smallmatrix}#1&#2\\#3&#4\end{smallmatrix}\right)}

\newtheorem{theorem}{Theorem}
\newtheorem{lemma}[theorem]{Lemma}
\newtheorem{prop}[theorem]{Proposition}
\newtheorem{cor}[theorem]{Corollary}

\newcommand     {\todo}[1]      {\begin{flushleft}\noindent{\tt - to be completed - #1}\end{flushleft}}

\numberwithin{theorem}{section}
\numberwithin{equation}{section}

\title{On generalized Maass relations for the Miyawaki-Ikeda lift}
\author[S.~Hayashida]{Shuichi Hayashida
}
\date{\today}

\begin{document}

\begin{abstract}
Some generalizations of the Maass relation for Siegel modular forms
of higher degrees have been obtained by several authors.
In the present article we first give a new generalization of the Maass relation
for Siegel-Eisenstein series of arbitrary degrees.
Furthermore, we show that the Duke-Imamoglu-Ibukiyama-Ikeda lifts satisfy
this generalized Maass relation with some modifications.
As an application of the generalized Maass relation in the present article
we give a new proof of the Miyawaki-Ikeda lifts
of two elliptic modular forms.
Namely, we compute the standard $L$-function
of the Miyawaki-Ikeda lifts of two elliptic modular forms
by using the generalized Maass relation.
\end{abstract}

\maketitle

\section{Introduction}\label{sec:intro}

\subsection{}
The Maass relation is a relation among Fourier coefficients of Siegel-Eisenstein
series of degree two, and the Maass relation characterizes
the Saito-Kurokawa lifts (cf. \cite{EZ}.)
In his article~\cite{Ya} Yamazaki has obtained a generalization of the Maass relation
for Siegel-Eisenstein series of arbitrary degrees.
Furthermore, in~\cite{Ya2} Yamazaki obtained a relation among Jacobi-Eisenstein series
of arbitrary degrees.
Here the Jacobi-Eisenstein series is a Jacobi form which is constructed like
a Siegel-Eisenstein series.
This relation among Jacobi-Eisenstein series is necessary to obtain a new
generalization of the Maass relation, which is different from the generalized Maass relation in~\cite{Ya}.
However, the relation among Jacobi-Eisenstein series in~\cite{Ya2} is not enough to
obtain a new generalization of the Maass relation,
because in~\cite{Ya2} the Jacobi-Eisenstein series of \textit{index $1$}
is treated and we need the relation among the Jacobi-Eisenstein series of \textit{arbitrary index}.
One of the aim of the present article is to generalize the relation among Jacobi-Eisenstein
series obtained in~\cite{Ya2} for arbitrary index and to give a new generalization of the Maass relation
for Siegel-Eisenstein series of general degrees.

On the other hand,
a generalization of the Saito-Kurokawa lift for Siegel modular forms of even degrees was
conjectured by Duke and Imamoglu, and by Ibukiyama, independently,
and the conjecture was solved by Ikeda~\cite{Ik}.
In the present article, we call these lifts the Duke-Imamoglu-Ibukiyama-Ikeda lifts.
It is known that the Duke-Imamoglu-Ibukiyama-Ikeda lifts satisfy the generalized Maass relations in~\cite{Ya}
by inserting the Satake parameters of the preimage of the Duke-Imamoglu-Ibukiyama-Ikeda lift into the relation (cf. \cite{FJlift}.)

By applying the Duke-Imamoglu-Ibukiyama-Ikeda lift,
Ikeda~\cite{Ik2} solved and generalized one of the two conjectures posed by Miyawaki~\cite{Mi}
under a certain assumption.
Namely, he obtained lifts from pairs of an elliptic modular form and a Siegel modular form of
degree $r$ to Siegel modular forms of degree $2n+r$
under the assumption that the constructed Siegel modular form does not vanish identically.
In the present article we call these lifts the Miyawaki-Ikeda lifts.
In~\cite{Ik2} Ikeda obtained a conjecture about the relation between
the Petersson norm of the Miyawaki-Ikeda lift and a special value of a certain $L$-function.
For more details about the conjecture of non-vanishing of the Miyawaki-Ikeda lift, we refer the reader to~\cite{Ik2}.

The purpose of the present article is as follows:
\begin{enumerate}
 \item[(1)]
  we generalize the relation among Jacobi-Eisenstein series given in \cite{Ya2}
  for arbitrary integer-indices
  and obtain a new generalization of the Maass relation for the Siegel-Eisenstein series
  of arbitrary degrees
  (Theorem~\ref{thm:eins}),
 \item[(2)]
  we show a new generalization of the Maass relation for the Duke-Imamoglu-Ibukiyama-Ikeda lifts
  (Theorem~\ref{thm:phi_maass}),
 \item[(3)]
  we obtain a new proof for the Miyawaki-Ikeda lifts of two elliptic modular forms
  by using the generalized Maass relations,
  namely, we give the expression of the standard $L$-function of the Miyawaki-Ikeda lifts
  of two elliptic modular forms
  (Theorem~\ref{thm:drei} and Corollary~\ref{coro:vier}.)
\end{enumerate}

As for generalization of the Maass relation,
Kohnen~\cite{Ko2} obtained another kind of generalization of the Maass relation
which is related to the Fourier-Jacobi coefficients of \textit{matrix index of size $2n-1$},
while the generalization of the Maass relation in the present article is related
to the Fourier-Jacobi coefficients of \textit{integer index}.
It is known that the generalized Maass relation in~\cite{Ko2}
characterizes the image of the Duke-Imamoglu-Ibukiyama-Ikeda lifts (cf. Kohnen-Kojima~\cite{KoKo}, Yamana~\cite{Ymn}.)

We remark that a certain identity of the spinor $L$-function of the Miyawaki-Ikeda lift
of two elliptic modular forms has been given by Heim~\cite{He}
for the case of degree three and weight twelve.
This identity has been generalized in~\cite{SpiM} for any odd degrees $2n-1$
and for any even weights $k$.

\subsection{}\label{ss:results}
We explain our results more precisely.
We denote by $\H_n$ the Siegel upper-half space of size $n$.
For integers $n$ and $k > n+2$, the Siegel-Eisenstein series of weight $k$ of degree $n+1$ is defined by
\begin{eqnarray*}
  E_{k}^{(n+1)}(Z)
 &:=&
  \sum_{M = \left( \begin{smallmatrix} A & B \\ C & D \end{smallmatrix}\right) \in \Gamma_{n+1,0}\backslash \Gamma_{n+1}}
          \det(C Z + D)^{-k},
\end{eqnarray*}
where $\tau \in \H_{n+1}$ and
where $\Gamma_{n+1} := \mbox{Sp}_{n+1}(\Z)$ is the symplectic group of size $2n+2$ with entries in $\Z$ and we set
$ \Gamma_{n+1,0} := \left\{ \begin{pmatrix} A & B \\ C & D \end{pmatrix} \in \Gamma_{n+1} \, | \, C = 0 \right\}$.
A Fourier-Jacobi expansion of $E_{k}^{(n+1)}$ is given by
\begin{eqnarray*}
  E_{k}^{(n+1)}\left(\begin{pmatrix} \tau & z \\ {^t z} & \omega \end{pmatrix}\right)
  &=&
  \sum_{m = 0}^{\infty} e_{k,m}^{(n)}(\tau,z) e^{2\pi i m \omega},
\end{eqnarray*}
where $\tau \in \H_{n}$, $\omega \in \H_{1}$ and $z \in \C^n$.
The form $e_{k,m}^{(n)}$ is called the $m$-th Fourier-Jacobi coefficient of $E_{k}^{(n+1)}$.
We remark that $e_{k,m}^{(n)}$ is a Jacobi form of weight $k$ of index $m$ of degree $n$
(cf. Ziegler~\cite{Zi}.)

We denote by $J_{k,m}^{(n)}$ the space of Jacobi forms of weight $k$ of index $m$
of degree $n$.
For the definition of Jacobi forms of higher degree, we refer the reader to \cite{Zi}
or Section~\ref{ss:Jacobi_group} in the present article.
We define two kinds of index-shift maps:
\begin{eqnarray*}
 V_{l,n-l}(p^2) \, : \, J_{k,m}^{(n)} \rightarrow J_{k,mp^2}^{(n)}, \\
 U(p) \, : \, J_{k,m}^{(n)} \rightarrow J_{k,mp^2}^{(n)}.
\end{eqnarray*}
Here the index-shift map $V_{l,n-l}(p^2)$ $(0 \leq l \leq n)$
is given by the action of the double coset $\Gamma_n \mbox{diag}(1_{l},p 1_{n-1}, p^2 1_{l}, p 1_{n-l}) \Gamma_n$.
For the precise definition of $V_{l,n-1}(p^2)$ see Section~\ref{ss:index_shift},
and we define $(\phi|U(d))(\tau,z) := \phi(\tau,d z)$ for $\phi \in J_{k,m}^{(n)}$ and
for any natural number $d$.

\begin{theorem}\label{thm:eins}
Let $e_{k,m}^{(n)}$ be the $m$-th Fourier-Jacobi coefficient of Siegel-Eisenstein series.
Then we obtain the relation
\begin{eqnarray*}
  &&
  e_{k,m}^{(n)}|\left(V_{0,n}(p^2), ..., V_{n,0}(p^2)\right) \\
  &=&
  \left(e_{k,\frac{m}{p^2}}^{(n)}|U(p^2),
          e_{k,m}^{(n)}|U(p),
          e_{k,mp^2}^{(n)} \right)
  \begin{pmatrix}  0 & 1 \\ p^{-k} & p^{-k}(-1+p\, \delta_{p|m}) \\ 0 & p^{-2k+2}\end{pmatrix}
  A_{2,n+1}^{p,k},
 \end{eqnarray*}
 where the both sides of the above identity are vectors of functions and
 $A_{2,n+1}^{p,k}$ is a certain matrix with size $2$ times $(n+1)$ which depends only on
 $p$ and $k$, and where we regard $e_{k,\frac{m}{p^2}}^{(n)}$ as identically $0$ if $p^2 {\not |} m$.
 Here $\delta_{p|m}$ is defined by 1 or 0, according as $p|m$ or $p {\not |} m$.
 For the precise definition of $A_{2,n+1}^{p,k}$, see Section~\ref{ss:Satake_Siegel_Phi}.
\end{theorem}
The relation in Theorem\ref{thm:eins} is a new generalization of the Maass relation for Siegel-Eisenstein series of arbitrary degrees.
As for the function $e_{k,m}^{(n)} | V_n(p)$, a similar identity has already been given in~\cite{Ya}.
Here the operator $V_n(p)$ is obtained from the double coset
$\Gamma_n \mbox{diag}(1_n,p 1_n) \Gamma_n$.

Now we apply the relation in Theorem~\ref{thm:eins} to the Duke-Imamoglu-Ibukiyama-Ikeda lifts.
We denote by $S_k(\Gamma_n)$ the space of Siegel cusp forms of weight $k$ of degree $n$.
Let $f \in S_{2k}(\Gamma_1)$ be a normalized Hecke eigenform
and let $F \in S_{k+n}(\Gamma_{2n})$ be a Duke-Imamoglu-Ibukiyama-Ikeda lift of $f$ (cf. Ikeda~\cite{Ik2}.)
We remark that there is no canonical choice of $F$, however $F$ is determined up to constant multiple.
We consider the Fourier-Jacobi expansion of $F$:
\begin{eqnarray*}
 F\left(\begin{pmatrix} \tau & z \\ {^t z} & \omega \end{pmatrix}\right)
 &=&
 \sum_{m=1}^{\infty} \phi_m(\tau,z) e^{2\pi i m \omega},
\end{eqnarray*}
where $\tau \in \H_{n}$, $\omega \in \H_{1}$ and $z \in \C^n$.
Then $\phi_m$ is the $m$-th Fourier-Jacobi coefficient of $F$
and is a Jacobi cusp form of weight $k+n$ of index $m$ of degree $2n-1$.
We denote by $J_{k,m}^{(n)\, cusp}$ the space of Jacobi cusp forms of
weight $k$ of index $m$ of degree $n$.
The restriction of the maps $V_{l,n-l}(p^2)$ and $U(p)$ to $J_{k,m}^{(n)\, cusp}$
gives the maps from $J_{k,m}^{(n)\, cusp}$ to $J_{k,mp^2}^{(n)\, cusp}$.
Let $\alpha_p^{\pm1}$ be the complex numbers which satisfy
\begin{eqnarray*}
    (\alpha_p + \alpha_p^{-1})p^{k-\frac12}
  &=&
    a(p),
\end{eqnarray*}
where $a(p)$ is the $p$-th Fourier coefficient of $f$.

The following theorem is a generalization of the Maass relation for the Duke-Imamoglu-Ibukiyama-Ikeda lifts,
which is different from the ones in~\cite{Ko2} and in~\cite{FJlift}.
\begin{theorem}\label{thm:phi_maass}
 Let $\phi_m \in J_{k+n,m}^{(2n-1)\, cusp}$ be the $m$-th Fourier-Jacobi coefficient of
 the Duke-Imamoglu-Ibukiyama-Ikeda lift $F$ as the above.
 Then we have
  \begin{eqnarray*}
    &&
    \phi_m | \left( V_{0,2n-1}(p^2), ... , V_{2n-1,0}(p^2)\right) \\
    &=&
    p^{-(n-1)(2k-1)}
    \left(
      \phi_{\frac{m}{p^2}} | U(p^2), 
      \phi_{m} | U(p),
      \phi_{mp^2}
    \right)
    \begin{pmatrix}
      0              & 1 \\
      p^{-k-n} & p^{-k-n}(-1+p\, \delta_{p|m})  \\
      0              & p^{-2k-2n+2}
    \end{pmatrix}
    A'_{2,2n}(\alpha_p),
  \end{eqnarray*}
 where $A'_{2,2n}(\alpha_p)$ is a certain matrix with size $2$ times $2n$ which depends
 only on $f$ and $p$.
 We regard the form $\phi_{\frac{m}{p^2}}$ as identically zero if $p^2 {\not |} m$.
 The matrix $A'_{2,2n}(\alpha_p)$ is obtained by substituting $X_p = \alpha_p$ into a matrix $A'_{2,2n}(X_p)$.
 For the precise definition of $A'_{2,2n}(X_p)$, see Section~\ref{ss:Satake_Siegel_Phi}.
\end{theorem}

Now we apply the relation in Theorem~\ref{thm:phi_maass} to the Miyawaki-Ikeda lifts
of two elliptic modular forms.
Let $f$ and $F$ be as above.
Let $g \in S_{k+n}(\Gamma_1)$ be a normalized Hecke eigenform.
Then one can construct a Siegel cusp form $\mathcal{F}_{f,g}$ of weight $k+n$ of degree $2n-1$:
\begin{eqnarray*}
 \mathcal{F}_{f,g}(\tau)
 &:=&
 \int_{\Gamma_1\backslash \mathfrak{H}_1}
   F\left(\begin{pmatrix} \tau & 0 \\ 0 & \omega \end{pmatrix}\right)
   \overline{g(\omega)}\, \mbox{Im}(\omega)^{k+n-2} \, d\, \omega.
\end{eqnarray*}
The form $\mathcal{F}_{f,g}$ is the Miyawaki-Ikeda lift of $g$ associated to $f$.
It is shown by Ikeda~\cite{Ik2} that if $\mathcal{F}_{f,g}$ is not identically zero,
then $\mathcal{F}_{f,g}$ is an eigenfunction for Hecke operators for
the Hecke pair $(\Gamma_{2n-1}, \mbox{Sp}_{2n-1}(\Qq))$.
Furthermore, the standard $L$-function of $\mathcal{F}_{f,g}$ is expressed as
a certain product of $L$-functions related to $f$ and $g$.
Now by virtue of Theorem~\ref{thm:phi_maass}, we obtain a new proof of these facts
by using the generalized Maass relations.

\begin{theorem}\label{thm:drei}
Let $\mathcal{F}_{f,g} \in S_{k+n}(\Gamma_{2n-1})$ be the Miyawaki-Ikeda lift of $g$ associated to $f$.
Then
\begin{eqnarray*}
 &&
 \mathcal{F}_{f,g} | \left(T_{0,2n-1}(p^2), ..., T_{2n-1,0}(p^2)\right)\\
 &=&
 p^{2nk+n-1}
 \left( p^{-k-n}, p^{-2k-2n+2} \lambda_g(p^2)\right) A'_{2,2n}(\alpha_p)\,
 \mathcal{F}_{f,g},
\end{eqnarray*}
where $T_{l,2n-1-l}(p^2)$ are Hecke operators (see Section~\ref{ss:index_shift})
and $A'_{2.2n}(\alpha_p)$ is the same matrix in Theorem~\ref{thm:phi_maass}.
Here $\lambda_g(p^2)$ is the eigenvalue of $g$ for $T_{1,0}(p^2)$.
\end{theorem}

We denote by $\beta_p^{\pm1}$ the complex numbers which satisfy:
\begin{eqnarray*}
    (\beta_p + \beta_p^{-1})p^{\frac{k+n-1}{2}}
  &=&
    b(p),
\end{eqnarray*}
where $b(p)$ is the $p$-th Fourier coefficient of $g$.
The adjoint $L$-function of $g$ is defined by
\begin{eqnarray*}
   L(s,g,\mbox{Ad})
 &:=&
   \prod_p \left\{ (1-p^{-s})(1-\beta_p^2\, p^{-s})(1-\beta_p^{-2}\, p^{-s}) \right\}^{-1}.
\end{eqnarray*}

\begin{cor}\label{coro:vier}
If $\mathcal{F}_{f,g}$ is not identically zero, then
the Satake parameter of $\mathcal{F}_{f,g}$ at prime $p$ is
\begin{eqnarray*}
 \left\{ \mu_1^{\pm1}, ..., \mu_{2n-1}^{\pm1} \right\}
 &=&
 \left\{
   \beta_p^{\pm 2},
   \alpha_p^{\pm1} p^{-n+\frac32} , \alpha_p^{\pm1} p^{-n + \frac52} ,
     ... , \alpha_p^{\pm1} p^{n - \frac32}
 \right\}.
\end{eqnarray*}
Furthermore, the standard $L$-function of $\mathcal{F}_{f,g}$ is
\begin{eqnarray*}
 L(s,\mathcal{F}_{f,g},\mbox{st})
 &=&
 L(s,g,\mbox{Ad}) \prod_{i=1}^{2n-2} L(s+k+n-1-i,f),
\end{eqnarray*}
where $L(s,f)$ is the Hecke $L$-function of $f$.
(see Section~\ref{ss:standard-L} for the definition of the standard $L$-function.)
\end{cor}
We remark that Theorem~\ref{thm:drei} follows from Corollary~\ref{coro:vier}.
And Corollary~\ref{coro:vier} has already been shown by Ikeda~\cite{Ik} for
more general case, namely for Siegel modular form $g \in S_{k+n}(\Gamma_r)$.
Here we obtained a new proof of Theorem~\ref{thm:drei} and Corollary~\ref{coro:vier}
by using the generalized Maass relation.

Furthermore, we remark that a certain identity of the spinor $L$-function of $\mathcal{F}_{f,g}$ has
been obtained in~\cite{SpiM}
which is a generalization of the case $(n,k)=(2,12)$ in~\cite{He}.

\vspace{5mm}

This paper is organized as follows:
In Section~\ref{sec:operators_Jacobi} we give a notation and review some
operators for Jacobi forms,
and in Section~\ref{sec:Jacobi-Eisenstein}
we shall show a certain relation
among Jacobi-Eisenstein series with respect to the index-shift maps.
In Section~\ref{sec:Fourier-Jacobi} we shall prove Theorem~\ref{thm:eins},
while we shall prove Theorem~\ref{thm:phi_maass}, Theorem~\ref{thm:drei}
and Corollary~\ref{coro:vier} in Section~\ref{sec:Miyawaki-Ikeda}.

\vspace{5mm}

Acknowledgement: to be entered later.


\section{Operators on Jacobi forms}\label{sec:operators_Jacobi}

\subsection{Symbols}

We denote by $M_{i,j}(R)$ the set of all matrices with entries in the ring $R$
and put $M_n(R) := M_{n,n}(R)$.
For any square matrix $A \in M_{n}(\Z)$ we denote by $\mbox{rank}_p(A)$
the rank of $A$ in $M_{n}(\Z/p\Z)$.
For any two matrices $A \in M_{n}(\Z)$ and $B \in M_{n,m}(\Z)$
we write $A[B]$ for ${^t B} A B$.
The set of all half-integral symmetric matrices of size $n$ is denoted by $\mbox{Sym}_{n}^*$.

We put $J_n := \begin{pmatrix}  0 & - 1_n \\ 1_n & 0 \end{pmatrix}$ and set
\begin{eqnarray*}
 \mbox{GSp}_n^{+}(\R)
 &:=&
 \left\{ M \in M_{2n}(\R) \, | \, M J_{n} {^t M} = \nu(g) J_{n}, \, \nu(g) > 0 \right\},
\end{eqnarray*}
where the number $\nu(g)$ is called the similitude of $g$.

We put $\Gamma_n := \mbox{Sp}_n(\Z) \subset \mbox{SL}_{2n}(\Z)$.
For any square matrix $x$ we set $e(x) := e^{2\pi i\, tr(x)}$, where $\mbox{tr}(x)$ denotes the trace of $x$.
For any natural number $m$ we put $< m > := \frac{m(m+1)}{2}$.

The symbol $\H_n$ denotes the Siegel upper-half space of size $n$.
The action of $\mbox{GSp}_n^{+}(\R)$ on $\H_n$ is given by
$g \cdot \tau := (A\tau+B)(C\tau+D)^{-1}$ for $g = \smat{A}{B}{C}{D} \in \mbox{GSp}_n^{+}(\R)$ and for $\tau \in \H_n$.

The symbol $\mbox{Hol}(\H_n\rightarrow \C)$ (resp. $\mbox{Hol}(\H_n \times \C^n \rightarrow \C)$)
denotes the space of all holomorphic function on $\H_n$ (resp. $\H_n \times \C^n$.)
For any integer $k$,
we define the slash operator $|_k$ :
\[
 (F|_k g)(\tau) := \det(C\tau+D)^{-k} F(g\cdot \tau),
\]
where $F \in \mbox{Hol}(\H_n\rightarrow \C)$, $g = \smat{A}{B}{C}{D} \in \mbox{GSp}_n^{+}(\R)$ and $\tau \in \H_n$.
By this definition the group $\mbox{GSp}_n^{+}(\R)$ acts on $\mbox{Hol}(\H_n\rightarrow \C)$.

\subsection{Jacobi group}\label{ss:Jacobi_group}
We define a subgroup of $\mbox{GSp}_{n+1}^{+}(\R)$:
\begin{eqnarray*}
 G_n^J &:=&
 \left\{
   \gamma \in \mbox{GSp}_{n+1}^{+}(\R)
   \, \left| \,
   \gamma =
    \left(
    \begin{smallmatrix}
      A & 0 & B & * \\
      * & \nu(\gamma) & * & * \\
      C & 0 & D & * \\
      0 & 0 & 0 & 1
    \end{smallmatrix}
    \right),\
    \begin{pmatrix}
      A & B \\ C & D
    \end{pmatrix}
    \in \mbox{GSp}_n^{+}(\R)
 \right\}.
 \right.
\end{eqnarray*}

A bijective map from $\mbox{GSp}_{n}^{+}(\R) \times \left(\R^n \times \R^n\right) \times \R$
to $G_n^J$
is given by
\begin{eqnarray*}
 \left[
   \begin{pmatrix} A&B\\C&D \end{pmatrix}, \left(\lambda, \mu\right), \kappa
 \right]
 \mapsto
   \left(\begin{smallmatrix}
        A & 0 & B & 0 \\
        0 & \nu(g) & 0 & 0 \\
        C & 0 & D & 0 \\
        0 & 0 & 0 & 1
    \end{smallmatrix} \right)
  \left(\begin{smallmatrix}
   1_n          & 0 & 0        & \mu \\
   {^t \lambda} & 1 & {^t \mu} & {^t \lambda}\mu + \kappa \\
   0            & 0 & 1_n      & -\lambda \\
   0            & 0 & 0        & 1
 \end{smallmatrix} \right),
\end{eqnarray*}
where $g = \begin{pmatrix} A & B \\ C & D \end{pmatrix} \in \mbox{GSp}_{n}^{+}(\R)$,
$\lambda,\mu \in \R^n$ and $\kappa \in \R$.
We identify $\mbox{GSp}_{n}^{+}(\R) \times \left(\R^n \times \R^n\right) \times \R$ and $G_n^J$.
By this bijection the group $G_n^J$ can be regarded as a semi-direct product of
$\mbox{GSp}_{n}^{+}(\R)$ and $\left(\left(\R^n \times \R^n\right) \times \R\right)$,
namely $G_n^J \cong \mbox{GSp}_{n}^{+}(\R) \ltimes \left(\left(\R^n \times \R^n\right) \times \R\right)$.

Let $k$ and $m$ be integers and 
let $\phi \in \mbox{Hol}(\H_n \times \C^n \rightarrow \C)$
be a holomorphic function on $\H_n \times \C^n$.
We define the slash operator $|_{k,m}$:
\begin{eqnarray*}
  (\phi|_{k,m}\gamma)(\tau,z)
 &:=&
  ((\phi(\tau,z) e(m\, \omega)) |_k \gamma)\, e(-\nu(\gamma)\, m\, \omega),
\end{eqnarray*}
where $\begin{pmatrix} \tau & z \\ ^t z & \omega \end{pmatrix} \in \H_{n+1}$,
 $\tau \in \H_n$, $\omega \in \H_1$, $z \in \C^n$
and $\gamma \in G_n^J$.
We remark that the RHS of the above definition does not depend
on the choice of $\omega$.
By this definition, the group $G_n^J$ acts on $\mbox{Hol}(\H_n \times \C^n \rightarrow \C)$.

For $\gamma = [g, (\lambda,\mu), \kappa] \in G_n^J$ we have
\begin{eqnarray*}
  (\phi|_{k,m}\gamma)(\tau,z)
 &=&
  \det(C\tau+D)^{-k}\, e(-\nu(g)\, m\, ((C\tau+D)^{-1}C)[z+ \tau \lambda+ \mu]) \\
 &&
  \times
  e(\nu(g)\, m\, (^t \lambda \tau \lambda + 2 {^t \lambda} z + 2 {^t \lambda}\mu + \kappa))\\
 &&
  \times
  \phi(g\cdot \tau, \nu(g) {^t (C\tau+D)^{-1}}(z+\tau\lambda+\mu)),
\end{eqnarray*}
where $g = \begin{pmatrix} A&B \\ C&D \end{pmatrix} \in \mbox{GSp}_n^{+}(\R)$.

We put a discrete subgroup $\Gamma_n^J$ of $G_n^J$:
\begin{eqnarray*}
 \Gamma_{n}^J
 &:=&
 \left\{ [M, (\lambda,\mu),\kappa] \in G_n^J \, | \, M \in \Gamma_n, \, (\lambda,\mu) \in \Z^n \times \Z^n,\,  \kappa \in \Z \right\}.
\end{eqnarray*}

We denote by $J_{k,m}^{(n)}$ the space of Jacobi forms
of weight $k$ of index $m$ of degree $n$ (cf. Ziegler~\cite{Zi}.)
For $n > 1$ the space $J_{k,m}^{(n)}$ is defined by
\begin{eqnarray*}
    J_{k,m}^{(n)}
  &:=&
    \left\{ \phi \in \mbox{Hol}(\H_n\times \C^n \rightarrow \C) \, |
            \, \phi|_{k,m}\gamma = \phi \mbox{ for any } \gamma \in \Gamma_{n}^J \right\}.
\end{eqnarray*}

\subsection{The standard $L$-functions}\label{ss:standard-L}
Let $F \in S_k(\Gamma_n)$ be a Siegel cusp form which is an eigenform for
all Hecke operators.
Let $\{\mu_{0,p},\mu_{1,p},...,\mu_{n,p} \}$ be the Satake parameter of $F$ at prime $p$.
The standard $L$-function of $F$ is defined by
\begin{eqnarray*}
    L(s,F,\mbox{st})
  &:=&
    \prod_{p} \left\{(1-p^{-s}) \prod_{i=1}^{n} (1-\mu_{i,p} p^{-s})(1-\mu_{i,p}^{-1} p^{-s})\right\}^{-1}.
\end{eqnarray*}
In our setting we have
$\mu_{0,p}^2 \mu_{1,p} \cdots \mu_{n,p} = p^{nk-<n>}$.

\subsection{index-shift maps of Jacobi forms}\label{ss:index_shift}
For any function $\phi \in J_{k,m}^{(n)}$ and for any matrix $g \in \mbox{GSp}_{n}^{+}(\R) \cap M_{2n}(\Z)$ we define
\begin{eqnarray*}
  \phi|V(\Gamma_n g \Gamma_n)
  &:=&
  \sum_{i} \phi|_{k,m} [g_i,(0,0),0],
\end{eqnarray*}
where $\Gamma_n g \Gamma_n = \displaystyle{\bigcup_i \Gamma_n g_i}$ is a coset decomposition.
It is known that $\phi|V(\Gamma_n g \Gamma_n) $ is well-defined and belongs to $J_{k,\nu(g) m}^{(n)}$.

For any integer $l$ $(0 \leq l \leq n)$,
we define
\begin{eqnarray*}
  \phi|V_{l,n-l}(p^2)
  &:=&
  \phi|V(\Gamma_n \mbox{diag}(1_{l}, p 1_{n-l}, p^2 1_{l}, p 1_{n-l}) \Gamma_n).
\end{eqnarray*}
For any non-negative integer $d$ we define
\begin{eqnarray*}
  (\phi|U(d))(\tau,z)
  &:=&
  \phi(\tau, d z).
\end{eqnarray*}
Then $\phi|V_{l,n-l}(p^2) \in J_{k,mp^2}^{(n)}$ and $\phi|U(d) \in J_{k,md^2}^{(n)}$.

Let $F$ be a Siegel modular form of weight $k$ of degree $n$.
Let $g$ be an element of $\mbox{GSp}_n^{+}(\R) \cap M_{2n}(\Z)$.
For any double coset $\Gamma_n g \Gamma_n$, the Hecke operator $T(\Gamma_n g \Gamma_n)$ is defined by
\begin{eqnarray*}
  F|T(\Gamma_n g \Gamma_n)
 &:=&
  \nu(g)^{nk-<n>} \sum_i F|_k g_i,
\end{eqnarray*}
where $\displaystyle{\Gamma_n g \Gamma_n = \bigcup_i \Gamma_n g_i}$ is a coset decomposition.
For any integer $l$ $(0 \leq l \leq n)$,
we define
\begin{eqnarray*}
  F|T_{l,n-l}(p^2)
  &:=&
  F|T(\Gamma_n \mbox{diag}(1_{l}, p 1_{n-l}, p^2 1_{l}, p 1_{n-l}) \Gamma_n).
\end{eqnarray*}

For any Jacobi form $\phi \in J_{k,m}^{(n)}$, we define the function
\begin{eqnarray*}
  W\!(\phi)(\tau)
 &:=&
  \phi(\tau,0)
\end{eqnarray*}
for $\tau \in \H_n$.
From the definition of Jacobi forms, it follows that $W\!(\phi)$ is
a Siegel modular form of weight $k$ of degree $n$.

Furthermore, due to the straightforward calculation, we obtain
\begin{eqnarray}\label{eq:W_T_V}
  W\!(\phi)|T(\Gamma_n g \Gamma_n)
 &=&
  \nu(g)^{nk-<n>}
  W\!(\phi|V(\Gamma_n g \Gamma_n))
\end{eqnarray}
for any Jacobi form $\phi \in J_{k,m}^{(n)}$ and for any
$g \in \mbox{GSp}_n^{+}(\R) \cap M_{2n}(\Z)$.

\subsection{Siegel $\Phi$-operator for Jacobi forms}\label{ss:Phi_Jacobi}

Let $\phi \in \mbox{Hol}(\H_n\times \C^n \rightarrow \C)$ be a holomorphic function.
We define the Siegel $\Phi$-operator:
\begin{eqnarray*}
 \Phi(\phi)\left(\tau_1, z_1 \right)
 :=
 \lim_{t \rightarrow + \infty}
   \phi\left(\begin{pmatrix} \tau_1 & 0 \\ 0 & i t \end{pmatrix}, \begin{pmatrix} z_1 \\ 0 \end{pmatrix} \right),
\end{eqnarray*}
where $\tau_1 \in \H_{n-1}$ and $z_1 \in \C^{n-1}$.

It is known that if $\phi \in J_{k,m}^{(n)}$ is a Jacobi form,
then the function $\Phi(\phi)$ is also a Jacobi form which belongs to $J_{k,m}^{(n-1)}$.

\subsection{The Satake isomorphism and the Siegel $\Phi$-operator}\label{ss:Satake_Siegel_Phi}

Let $\mathcal{H}_p^n$ be the local Hecke ring with respect to the Hecke pair
$(\Gamma_n, \mbox{GSp}_n^+(\R)\cap M_{2n}(\Z[p^{-1}]))$.
We denote by
 $\C[X_0^{\pm 1}, ... , X_n^{\pm1}]^{W_n}$
the subring of the polynomial ring $\C[X_0^{\pm 1}, ... , X_n^{\pm1}]$
which is invariant under the action of the Weyl group $W_n$ associated to the symplectic group.
The Satake isomorphism $\varphi_n \, : \, \mathcal{H}_p^n \rightarrow \C[X_0^{\pm 1}, ... , X_n^{\pm1}]^{W_n}$
is given by
\begin{eqnarray*}
 \Gamma_n g \Gamma_n =
 \bigcup_i \Gamma_n \begin{pmatrix} p^{l}\, {^t D_i}^{-1} & B_i \\ 0 & D_i \end{pmatrix}
 \mapsto
 X_0^{l }\sum_i  \prod_j \left( \frac{X_j}{p^j} \right)^{l_{i,j}},
\end{eqnarray*}
where
$\nu(g) = p^{l}$
and
$D_i
 = 
 \begin{pmatrix} p^{l_{i,1}} & * & * \\
                 & \ddots & * \\ & & p^{l_{i,n}} \end{pmatrix}$
(cf. Andrianov~\cite{An}.)

We write $\varphi = \varphi_n$ for simplicity.
In this article we consider the subring of $\mathcal{H}_p^n$ which is generated by
$T_{0.n}(p^2)^{\pm1}$ and $T_{l,n-l}(p^2)$ $(l = 1,...,n)$.

The following proposition follows from~\cite[Satz]{Kr}.
\begin{prop}\label{prop:krieg}
If $n \geq 2$ we have
\begin{eqnarray*}
 \varphi(T_{n,0}(p^2))
 &=&
 X_n \left\{\left(X_{n}^{-1} + (p-1)p^{-1} + X_{n}\right) \varphi(T_{n-1,0}(p^2)) \right. \\
 &&
 + \left. (p^2 - 1) p^{-1} \varphi(T_{n-2,1}(p^2))
 \right\}, \\
 \varphi(T_{1,n-1}(p^2))
 &=&
 X_n \left\{ p^{1-n}\, \varphi(T_{1,n-2}(p^2)) \right. \\
 &&
 + \left. \left( X_{n}^{-1} + (p-1)p^{-n} + X_{n} \right) \varphi(T_{0,n-1}(p^2)) \right\},\\
 \varphi(T_{0,n}(p^2))
 &=&
  X_n \left\{ p^{-n} \, \varphi(T_{0,n-1}(p^2)) \right\},
\end{eqnarray*}
and for $1 < j < n$ we have 
\begin{eqnarray*}
 \varphi(T_{j,n-j}(p^2))
 &=&
 X_{n} \left\{
 p^{j-n}  \varphi(T_{j,n-j-1}(p^2)) \right. \\
 &&
 + \left( X_{n}^{-1} + p^{j-n-1}(p-1) + X_{n} \right) \varphi(T_{j-1,n-j}(p^2)) \\
 &&
 + \left. (p^{2n-2j+2}-1) p^{j-n-1} \varphi(T_{j-2,n-j+1}(p^2)) \right\}.
\end{eqnarray*}
\end{prop}
\begin{proof}
We obtain this proposition by replacing $p^{-r}$ in~\cite[Satz]{Kr} by $p^{-n} X_n$.
For the detail the reader is referred to \cite[Satz]{Kr}.
\end{proof}

Now for integers $l$ $(2 \leq l)$, $t$ $(0 \leq t \leq l)$, $j$ $(0 \leq j \leq l)$, we put
\begin{eqnarray*}
 b_{t,j} := b_{t,j,l,p}(X_l) =
 \begin{cases}
  (p^{2l-2j+2}-1)p^{j-1-l} X_l & \mbox{ if } t = j - 2,\\
  1 + p^{j-1-l}(p-1) X_l + X_l^2 & \mbox{ if } t = j - 1,\\
  p^{-l+j} X_l & \mbox{ if } t = j, \\
  0 & \mbox{ otherwise},
 \end{cases}
\end{eqnarray*}
and we put a matrix
\begin{eqnarray*}
  B_{l,l+1}(X_l)
   &:=& \left(b_{t,j} \right)_{\begin{smallmatrix} t = 0,...,l-1 \\ j = 0,...,l \end{smallmatrix}}
    = \begin{pmatrix} b_{0,0} & \cdots &b_{0,l} \\
                                    \vdots & \cdots & \vdots \\
                                    b_{l-1,0} & \cdots & b_{l-1,l} \end{pmatrix}
\end{eqnarray*}
with entries in $\C[X_l,X_l^{-1}]$.
From Proposition~\ref{prop:krieg} and from the definition of $B_{l,l+1}(X_l)$,
we have the identity:
\begin{eqnarray*}
 (\varphi(T_{0,l}(p^2)) , ... , \varphi(T_{l,0}(p^2)))
 &=&
 (\varphi(T_{0,l-1}(p^2)), ... , \varphi(T_{l-1,0}(p^2)))
 B_{l,l+1}(X_l).
\end{eqnarray*}

For Jacobi forms we obtain the following lemma.
\begin{lemma}\label{lem:Phi_Tauschen}
Let $\phi$ $\in$ $J_{k,m}^{(l)}$ be a Jacobi form
such that $\Phi(\phi)$ is not identically zero.
Then we have
\begin{eqnarray*}
    \Phi(\phi|(V_{0,l}(p^2),...,V_{l,0}(p^2)))
  &=&
    \left(\Phi(\phi)|(V_{0,l-1}(p^2),...,V_{l-1,0}(p^2)) \right) B_{l,l+1}(p^{l-k}),
\end{eqnarray*}
where we put $\phi|(V_{0,l}(p^2),...,V_{l,0}(p^2)) := (\phi|V_{0,l}(p^2),...,\phi|V_{l,0}(p^2))$.
\end{lemma}
\begin{proof}
Let $\gamma = \left[\begin{pmatrix} A&B \\ 0 & D \end{pmatrix}, (0,0),0 \right] \in G_l^J$
with $A = \begin{pmatrix} A^*& 0 \\ \mathfrak{a}  & a \end{pmatrix}$,
$B = \begin{pmatrix} B^*& \mathfrak{b}_1 \\ \mathfrak{b}_2 & b \end{pmatrix}$,
$D = \begin{pmatrix} D^*& \mathfrak{d} \\ 0 & d \end{pmatrix}$,
where $A^*$, $D^* \in \mbox{GL}_{l-1}(\R)$ and $B^* \in M_{l-1}(\R)$.
Then
\begin{eqnarray*}
    \Phi( \phi |_{k,m} \gamma)
  &=&
   d^{-k} \Phi(\phi) |_{k,m} \gamma^*,
\end{eqnarray*}
where $\gamma^* = \left[\begin{pmatrix} A^*& B^* \\ 0 & D^* \end{pmatrix},(0,0),0 \right]
\in G_{l-1}^J$.

The rest of the proof of this lemma is the same to the case of Siegel modular forms (cf.~\cite[Satz]{Kr}.)
Thus we conclude this lemma.
\end{proof}

We define a matrix
\begin{eqnarray*}
    B_{2,n+1}(X_2,X_3,...,X_n)
  &:=& \displaystyle{\prod_{l=2}^n B_{l,l+1}(X_l)},
\end{eqnarray*}
which entries are in $\C[X_2^{\pm},...,X_n^{\pm}]$.
Then we have
\begin{eqnarray*}
  (\varphi(T_{0,n}(p^2)) , ... , \varphi(T_{n,0}(p^2)))
  &=&
  (\varphi(T_{0,1}(p^2)),\varphi(T_{1,0}(p^2))) B_{2,n+1}(X_2,...,X_n).
\end{eqnarray*}

The precise expression of $\varphi(T_{l,n-l}(p^2))$ by using the elementary symmetric polynomials
has been given in~\cite[Korollar 2]{Kr}.

To explain our results we define two matrices $A_{2,n+1}^{p,k}$ and $A'_{2,2n}(X_p)$.
First we define a $2 \times (n+1)$ matrix
\begin{eqnarray*}
    A_{2,n+1}^{p,k}
 &:=&
    B_{2,n+1}(p^{2-k},p^{3-k},...,p^{n-k}).
\end{eqnarray*}
We remark that the matrix $A_{2,n+1}^{p,k}$ 
depends only on the prime $p$ and the integer $k>0$.

We set a $2\times 2n$ matrix
\begin{eqnarray*}
 B'_{2,2n}(X_2,...,X_{2n-1})
 &:=&
 \left(\prod_{i=2}^{2n-1} X_i\right)^{-1} B_{2,2n}(X_2,...,X_{2n-1}).
\end{eqnarray*}
From the definition of $B_{2,2n}(X_2,...,X_{2n-1})$
it is not difficult to see that the entries in
the matrix $B'_{2,2n}(X_2,...,X_{2n-1})$
belong to $\C[X_2+X_2^{-1},...,X_{2n-1}+X_{2n-1}^{-1}]$.
We define a $2\times 2n$ matrix
 \begin{eqnarray*}
     A'_{2,2n}(X_p)
   &:=&
     B'_{2,2n}(p^{\frac32-n} X_p, p^{\frac52-n} X_p ,..., p^{-\frac32+n} X_p).
 \end{eqnarray*}
In Section~\ref{ss:pr_of_thm2} we will show $A'_{2,2n}(X_p) = A'_{2,2n}(X_p^{-1})$.

\section{Jacobi-Eisenstein series}\label{sec:Jacobi-Eisenstein}

The goal of this section is to show a certain relation
among Jacobi-Eisenstein series with respect to
the index-shift maps $V_{l,n-l}(p^2)$ $(l=0,...,n)$.

\subsection{Definition of Jacobi-Eisenstein series} 

For integers $k$, $m$ and $n$, we define the Jacobi-Eisenstein series
of weight $k$ of index $m$ of degree $n$ by
\begin{eqnarray*}
  E_{k,m}^{(n)}(\tau,z)
  &:=&
  \sum_{\gamma \in \Gamma_{n,0}^J \backslash \Gamma_{n}^J}
    (1|_{k,m} \gamma),
\end{eqnarray*}
where we put
\begin{eqnarray*}
 \Gamma_{n,0}^J
  &:=&
 \left\{
  \left(\begin{smallmatrix} A & 0 & B & 0 \\ 0 & 1 & 0 & 0 \\ 0 & 0 & D & 0 \\ 0 & 0 & 0 & 1 \end{smallmatrix} \right)
  \left(\begin{smallmatrix} 1_n & 0 & 0 & \mu \\ 0 & 1 & {^t \mu} & \kappa \\ 0 & 0 & 1_n & 0 \\ 0 & 0 & 0 & 1 \end{smallmatrix} \right)
 \in \Gamma_{n}^J \, \left| \, \begin{pmatrix} A & B \\ 0 & D \end{pmatrix} \in \Gamma_n, \mu \in \Z^n, \kappa \in \Z \right. \right\}.
\end{eqnarray*}
It is known that if $k > n + 2$, then $E_{k,m}^{(n)}$ converges and belongs to $J_{k,m}^{(n)}$ (cf. Ziegler~\cite{Zi}.)

The purpose of this section is to show that $E_{k,m}^{(n)}|V_{l,n-l}(p^2)$
is a linear combination of three forms $E_{k,\frac{m}{p^2}}^{(n)}|U(p^2)$,
$E_{k,m}^{(n)}|U(p)$ and $E_{k,mp^2}^{(n)}$.

\begin{lemma}\label{lemma:linear_independent}
Let $m$ and $n$ be positive integers.
Then the forms $\left\{ E_{k,\frac{m}{d^2}}^{(n)}|U(d) \right\}_d$ are linearly independent,
where $d$ runs over all positive integers such that $d^2 | m$.
\end{lemma}
\begin{proof}
Let $\Phi$ be the Siegel $\Phi$-operator
for Jacobi forms introduced in Section~\ref{ss:Phi_Jacobi}.
It follows from the definition that $\Phi(E_{k,m}^{(n)}) = E_{k,m}^{(n-1)}$.
Hence it is enough to show that
the forms $\left\{ E_{k,\frac{m}{d^2}}^{(1)}|U(d) \right\}_d$
are linearly independent.

Let $E_{k,m}^{(1)}(\tau,z) = \sum_{n',r} c(n',r)\, e(n'\tau + r z)$ be the Fourier expansion
of $E_{k,m}^{(1)}$.
We call $c(n',r)$ the $(n',r)$-th Fourier coefficient of $E_{k,m}^{(1)}$.
Let $n' > 0$ and $r \geq 0$ be integers such that $4n' m - r^2 > 0$.
Then it is known that the $(n',r)$-th Fourier coefficient of $E_{k,m}^{(1)}$ is not zero
(cf. Eichler-Zagier~\cite[p.17--p.20]{EZ}.)
On the other hand, for any $d > 1$ such that $d^2 | m$,
the $(n',r)$-th Fourier coefficient
of $E_{k,\frac{m}{d^2}}^{(1)}|U(d)$ is zero unless $d|r$.
Therefore we obtain this lemma.
\end{proof}

\subsection{Definition of a form $K_{i,j}^{(n)}$}\label{ss:def_Kij}

We quote some symbols from~\cite{Ya2}.
For a fixed prime $p$ and for $0 \leq i \leq j \leq n$, we put
\begin{eqnarray*}
 \delta_{i,j} := \mbox{diag}(1_{i}, p 1_{j-i}, p^2 1_{n-j})
\end{eqnarray*}
and
\begin{eqnarray*}
 \delta_{i} := \delta_{i,n} = \mbox{diag}(1_{i}, p 1_{n-i}).
\end{eqnarray*}
And for $x = \mbox{diag}(0_{i}, x_{2,2}, 0_{n-j})$ with $x_{2,2} = {^t x_{2,2}} \in M_{j-i}(\Z)$ we set 

\begin{eqnarray*}
  \delta_{i,j}(x) := \begin{pmatrix} p^2 \delta_{i,j}^{-1} & x \\ 0_n & \delta_{i,j} \end{pmatrix}.
\end{eqnarray*}

We denote by $\Gamma_{n,0}$ the set of all matrices $\begin{pmatrix} A & B \\ 0_n & D \end{pmatrix}$
in $\Gamma_n$.
We set
\begin{eqnarray*}
 \Gamma(\delta_{i,j})
 &:=&
 \left. \left\{
    \begin{pmatrix}
        A   & B \\
        0_n & D
    \end{pmatrix} \in \Gamma_{n,0}
    \, \right| \,
    A \in \delta_{i,j} \mbox{GL}_n(\Z) \, \delta_{i,j}^{-1}
 \right\},\\
  \Gamma(\delta_{i})
 &:=&
 \left. \left\{
    \begin{pmatrix}
        A   & B \\
        0_n & D
    \end{pmatrix} \in \Gamma_{n,0}
    \, \right| \,
    A \in \delta_{i} \mbox{GL}_n(\Z) \, \delta_{i}^{-1}
 \right\}
\end{eqnarray*}
and put a subgroup $\Gamma(\delta_{i,j}(x))$ of $\Gamma(\delta_{i,j})$:
\begin{eqnarray*}
 \Gamma(\delta_{i,j}(x))
 &:=&
 \Gamma_n \cap (\delta_{i,j}(x)^{-1} \Gamma_{n,0}\, \delta_{i,j}(x)).
\end{eqnarray*}

For $\lambda \in \Z^n$ and for $M \in \mbox{GSp}_n^+(\R)\cap M_{2n}(\Z)$ we put
\begin{eqnarray*}
 j(k,m;M,\lambda)(\tau,z)
   &:=&
 (1|_{k,m}[1_{2n},(\lambda,0),0] [M,(0,0),0])(\tau,z).
\end{eqnarray*}

For two matrices $x = \mbox{diag}(0_i,x_{2,2},0_{n-j})$ and $y = \mbox{diag}(0_j,y_{2,2},0_{n-j})$ such that
$x_{2,2} = {^t x_{2,2}}$, $y_{2,2} = {^t y_{2,2}}$ $\in M_{j-i}(\Z)$,
we say they are equivalent and write $[x] = [y]$,
if there exists a matrix
$
   u =
    \begin{pmatrix}
      u_{1,1}     & u_{1,2} & u_{1,3} \\
      p\, u_{2,1} & u_{2,2} & u_{2,3} \\
      p^2 u_{3,1} & p\, u_{3,2} & u_{3,3}
    \end{pmatrix} \in \delta_{i,j}\mbox{GL}_n(\Z) \delta_{i,j}^{-1} \cap \mbox{GL}_n(\Z)
$
which satisfies $u_{2,2} x_{2,2} {^t u_{2,2}} \equiv y_{2,2} \!\! \mod$ $p$,
where $u_{2,2} \in M_{j-i}(\Z)$, $u_{1,1} \in M_i(\Z)$ and $u_{3,3} \in M_{n-j}(\Z)$.

We define a function $K_{i,j}^{\alpha}$ on $(\tau,z) \in \H_n \times \C^n$ by
\begin{eqnarray*}
 K_{i,j}^{\alpha}
 &:=&
 K_{i,j,m,p}^{\alpha}(\tau,z)
 \ = \
   \sum_{\begin{smallmatrix} [x] \\ rank_p(x) = \alpha \end{smallmatrix}}
   \sum_{M \in \Gamma(\delta_{i,j}(x))\backslash \Gamma_n}
   \sum_{\lambda \in \Z^n} j(k,m; \delta_{i,j}(x) M,\lambda)(\tau,z),
\end{eqnarray*}
where in the first summation in the RHS, $[x]$ runs over all equivalence classes
which satisfy $\mbox{rank}_p(x) = \alpha$.

\begin{prop}[Yamazaki \cite{Ya2}]\label{prop:Ya2_decomposition}
 The double coset $\Gamma_n \begin{pmatrix} \delta_l & 0_n \\ 0_n & p^2 \delta_l^{-1} \end{pmatrix} \Gamma_n$
 is a disjoint union
 \begin{eqnarray*}
     \Gamma_n \begin{pmatrix} \delta_l & 0_n \\ 0_n & p^2 \delta_l^{-1} \end{pmatrix} \Gamma_n
   &=&
     \bigcup_{\begin{smallmatrix} i,j \\ 0 \leq i \leq j \leq n \end{smallmatrix}}
     \bigcup_{\begin{smallmatrix} [x] \\ rank_p(x) = l-n-i+j \end{smallmatrix}}
      \Gamma_{n,0} \delta_{i,j}(x) \Gamma_n,
 \end{eqnarray*}
 where in the last union in the RHS, $[x]$ runs over all equivalence classes
 which satisfy $\mbox{rank}_p(x) = l-n-i+j$.
\end{prop}
\begin{proof}
This proposition has been shown in~\cite[Corollary 2.2]{Ya2}.
\end{proof}

\begin{lemma}\label{lem:E_K}
We obtain
\begin{eqnarray*}
   E_{k,m}^{(n)} | V_{l,n-l}(p^2)
 &=&
   \sum_{\begin{smallmatrix} i,j \\ 0 \leq i \leq j \leq n \end{smallmatrix}}
    K_{i,j}^{l-i-n+j}.
\end{eqnarray*}
\end{lemma}
\begin{proof}
It follows from Proposition~\ref{prop:Ya2_decomposition}
and from the definitions of $E_{k,m}^{(n)}$, $V_{l,n-l}(p^2)$ and $K_{i,j}^{\alpha}$.
\end{proof}

\begin{lemma}\label{lem:kij_yam}
If $p^2 |m$, then
\begin{eqnarray*}
 K_{i,j}^{\alpha}
 &=&
   p^{-k(2n-i-j)+(n-j)(n-i+1)}
   \sum_{\begin{smallmatrix}
          x = diag(0_i,x_{2,2},0_{n-j}) \\
          x_{2,2} = {^t x_{2,2}} \in M_{j-i}(\Z)\, mod\, p \\
          rank_p(x_{2,2}) = \alpha \end{smallmatrix}}
   \sum_{M \in \Gamma(\delta_{i,j}) \backslash \Gamma_n} \\
   && \times 
   \sum_{\lambda \in (p^2 \Z)^i \times (p \Z)^{j-i} \times \Z^{n-j}}
   j\!\left(k,\frac{m}{p^2}; \begin{pmatrix} 1_n & p^{-1} x \\ 0 & 1_n \end{pmatrix} M , \lambda\right)\!(\tau,p^2z).
\end{eqnarray*}

If $p^2 {\not |} m$, then
\begin{eqnarray*}
 K_{i,j}^{\alpha}
 &=&
   p^{-k(2n-i-j)+(n-j)(n-i+1)}
   \sum_{\begin{smallmatrix}
          x = diag(0_i,x_{2,2},0_{n-j}) \\
          x_{2,2} = {^t x_{2,2}} \in M_{j-i}(\Z)\, mod\, p  \\
          rank_p(x_{2,2}) = \alpha \end{smallmatrix}}
   \sum_{M \in \Gamma(\delta_{i,j}) \backslash \Gamma_n} \\
   && \times 
   \sum_{\lambda \in (p \Z)^i \times \Z^{n-i}}
   j\!\left(k,m; \begin{pmatrix} 1_n & p^{-1} x \\ 0 & 1_n \end{pmatrix} M , \lambda\right)\!(\tau,p z).
\end{eqnarray*}
We remark that this lemma has been shown for the case $m=1$ by Yamazaki~\cite{Ya2}.
\end{lemma}
\begin{proof}
The proof of this lemma is an analogue to~\cite{Ya2} and straightforward.
If $p^2 {\not |} m$, then the proof is similar to the case $m=1$.
Hence we assume $p^2 | m$ and shall  prove this lemma.

We put
$U
 :=
 \left\{\left(\begin{smallmatrix}1_n & s \\ 0 & 1_n \end{smallmatrix}\right) \, | \,
 s = ^t s \in M_n(\Z)\right\}$.
Then the set
\[
 U' :=
\left\{ \left(\begin{smallmatrix} 1_n & s \\ 0 & 1_n \end{smallmatrix}\right) \, \left| \,
   s = \left(\begin{smallmatrix} 0 & 0 & 0 \\ 0 & 0 & s_{23} \\ 0 & ^t s_{23} & s_{33} \end{smallmatrix}\right) \!\!\!\!\mod p,\
   s_{23} \in M_{j-i,n-j}(\Z) ,\
   s_{33} = {^t s_{33}} \in M_{n-j}(\Z)
 \right\} \right.
\]
is a complete set of representatives of
$\Gamma(\delta_{i,j}(x))\backslash \Gamma(\delta_{i,j}(x))U$.
Thus
\begin{eqnarray*}
  &&
   \sum_{\begin{smallmatrix} [x] \\ rank_p(x) = \alpha \end{smallmatrix}}
   \sum_{M \in \Gamma(\delta_{i,j}(x))\backslash \Gamma_n}
   \sum_{\lambda \in \Z^n} j(k,m; \delta_{i,j}(x) M,\lambda)(\tau,z) \\
  &=&
   \sum_{\begin{smallmatrix} [x] \\ rank_p(x) = \alpha \end{smallmatrix}}
   \sum_{M \in \Gamma(\delta_{i,j}(x)) U \backslash \Gamma_n}
   \sum_{\lambda \in \Z^n}
     j(k,m:\delta_{i,j}(x)M,\lambda)(\tau,z)
    \!\!
   \sum_{\left(\begin{smallmatrix}1_n & s \\ 0 & 1_n \end{smallmatrix}\right) \in U'}
    \!\!
    e(p^2 m {^t \lambda} \delta_{i,j}^{-1} s \delta_{i,j}^{-1} \lambda)\\
  &=&
   p^{(j-i)(n-j)+ (n-j)(n-j+1)}
   \sum_{\begin{smallmatrix} [x] \\ rank_p(x) = \alpha \end{smallmatrix}}
   \sum_{M \in \Gamma(\delta_{i,j}(x)) U \backslash \Gamma_n}
   \sum_{\lambda \in \Z^n}
     j(k,m:\delta_{i,j}(x)M,\lambda)(\tau,z).
\end{eqnarray*}
We remark
\begin{eqnarray*}
    j(k,m:\delta_{i,j}(x),\lambda)(\tau,z)
  &=&
    p^{-k(2n-i-j)}e\!\left(m{^t \lambda}(p^2\delta_{i,j}^{-1}\tau\delta_{i,j}^{-1}+ p^{-1} x)\lambda
                    + 2p^2 m {^t \lambda} \delta_{i,j}^{-1}z \right).
\end{eqnarray*}
Hence if we put $\lambda' = p^2 \delta_{i,j}^{-1} \lambda$,
then $\lambda' \in (p^2 \Z)^i \times (p\Z)^{j-i} \times \Z^{n-j}$
and we have
\begin{eqnarray*}
    j(k,m:\delta_{i,j}(x),\lambda)(\tau,z)
  &=&
    p^{-k(2n-i-j)} j(k, p^{-2}m :\left(\begin{smallmatrix} 1_n & p^{-1} x \\ 0 & 1_n \end{smallmatrix}\right), \lambda').
    (\tau,p^2 z).
\end{eqnarray*}
Thus
\begin{eqnarray*}
   K_{i,j}^{\alpha}
  &=&
   p^{-k(2n-i-j)+(n-j)(n-i+1)}
   \sum_{\begin{smallmatrix} [x] \\ rank_p(x) = \alpha \end{smallmatrix}}
   \sum_{M \in \Gamma(\delta_{i,j}(x)) U \backslash \Gamma_n}\\
  &&
   \times
   \sum_{\lambda' \in (p^2 \Z)^i \times (p\Z)^{j-i} \times \Z^{n-j}}
     j(k, p^{-2}m :\left(\begin{smallmatrix} 1_n & p^{-1} x \\ 0 & 1_n \end{smallmatrix}\right)M, \lambda')(\tau,p^2 z)\\
  &=&
   p^{-k(2n-i-j)+(n-j)(n-i+1)}
   \sum_{\begin{smallmatrix} [x] \\ rank_p(x) = \alpha \end{smallmatrix}}
   \sum_{\left(\begin{smallmatrix} u & 0 \\ 0 & {^t u}^{-1} \end{smallmatrix}\right)
           \in \Gamma(\delta_{i,j}(x)) U \backslash \Gamma(\delta_{i,j})}
   \sum_{M \in \Gamma(\delta_{i,j}) \backslash \Gamma_n}\\
  &&
   \times
   \sum_{\lambda' \in (p^2 \Z)^i \times (p\Z)^{j-i} \times \Z^{n-j}}
     j(k, p^{-2}m :\left(\begin{smallmatrix} 1_n & p^{-1} u^{-1}x{^t u}^{-1} \\ 0 & 1_n \end{smallmatrix}\right)M,
       {^t u}\lambda')(\tau,p^2 z).
\end{eqnarray*}
Here, the matrix $u$ in the above summation belongs to $\delta_{i,j} \mbox{GL}(n,\Z) \delta_{i,j}^{-1} \cap \mbox{GL}(n,\Z)$.
Hence ${^t u}$ stabilizes the lattice $(p^2 \Z)^i \times (p\Z)^{j-i} \times \Z^{n-j}$.
Furthermore, the summation over the equivalent class $[x]$ and the summation over the representatives
of $\Gamma(\delta_{i,j}(x)) U \backslash \Gamma(\delta_{i,j})$
turns into the summation over $x = \mbox{diag}(0,x_{2,2},0)$ such that
$x_{2,2} = {^t x_{2,2}} \in M_{j-i}(\Z) \!\!\mod p$ and $\mbox{rank}_p(x) = \alpha$.
Therefore we conclude this lemma.
\end{proof}

\subsection{Summation $ G_j^n(m, \lambda)$}

We define
\begin{eqnarray*}
 g_p(n,i) &:=&
 \begin{cases}
   \prod_{a = 1}^i (p^{n-a+1}-1) (p^a - 1)^{-1} & \mbox{ if } 1 \leq i \leq n,\\
   1 & \mbox{ if } i = 0,\\
   0 & \mbox{ otherwise}.
 \end{cases}
\end{eqnarray*}

For any $\lambda \in \Z^n$ and for $0 \leq j \leq n$ we define
\begin{eqnarray*}
 G_j^n(m, \lambda)
 &:=&
 \sum_{\begin{smallmatrix} x = ^t x \in M_{n}(\Z/p\Z) \\ rank_p x = j\end{smallmatrix}}
  e\left(\frac{m}{p} {^t \lambda} x \lambda \right) .
\end{eqnarray*}

\begin{prop}\label{prop:gauss}
For $m \in \Z$ and for $\lambda \in \Z^n$ we have
\begin{eqnarray*}
 G_j^n(m, \lambda)
 &=&
 \begin{cases}
  p^{\lfloor \frac{j}{2} \rfloor \left(\lfloor \frac{j}{2} \rfloor + 1 \right)}
  g_p(n,j)
  \displaystyle{\prod_{\begin{smallmatrix} \alpha = 1 \\ \alpha\, : \, odd \end{smallmatrix}}^j
  (p^\alpha - 1)}
  & \mbox{ if } m \lambda \equiv 0 \mod p, \\
  (-1)^j p^{\lfloor \frac{j}{2} \rfloor \left( \lfloor \frac{j}{2} \rfloor + 1 \right)}
  g_p\!\left(n-1, 2 \lfloor \frac{j}{2} \rfloor \right)
  \displaystyle{
   \prod_{\begin{smallmatrix} \alpha = 1 \\ \alpha\, :\, odd \end{smallmatrix}}^{j-1}
   (p^\alpha - 1)
  }
  & \mbox{ if } m \lambda \not \equiv 0 \mod p.
 \end{cases}
\end{eqnarray*}
\end{prop}
\begin{proof}
If $p|m$, then $ G_j^n(m,\lambda) = G_j^n(1,0)$.
And if $p {\not |} m$, then
$  G_j^n(m,\lambda) = G_j^n(1,\lambda)$.
Hence we need to calculate the case $m=1$.
The calculation of $G_j^n(1,\lambda)$ has already been obtained by \cite[Lemma 3.1]{Ya2}.
\end{proof}

\subsection{Some cardinalities}
In this subsection we will give some lemmas to
calculate $K_{i,j}^{\alpha}$.

For $0 \leq i \leq j \leq n$,
we put
\begin{eqnarray*}
 H_i &:=& \delta_{i} \mbox{GL}_n(\Z) \delta_{i}^{-1} \cap \mbox{GL}_n(\Z), \\
 H_{i,j} &:=&  \delta_{i,j} \mbox{GL}_n(\Z) \delta_{i,j}^{-1} \cap \mbox{GL}_n(\Z).
\end{eqnarray*}
We define two sets
\begin{eqnarray*}
  S_{i}
 &:=& \left. \left\{ \begin{pmatrix} * & * \\ p\, {^t b} & * \end{pmatrix}^{-1} \in \mbox{GL}_n(\Z)
              \, \right| \,
              b \in \Z^i
      \right\}, \\
  S_{i,j}
 &:=& \left. \left\{ \begin{pmatrix}       *  &       * & * \\
                                      p^2 {^t b_1} & p\, {^t b_2} & *\\
                     \end{pmatrix}^{-1} \in \mbox{GL}_n(\Z)
              \, \right| \,
              b_1 \in \Z^i, \, b_2 \in \Z^{j-i} \right\},
\end{eqnarray*}
where $b$, $b_1$ and $b_2$ in the above sets are column vectors.



\begin{lemma}\label{lem:value_a}
We have
\begin{eqnarray*}
 \left| H_i \backslash \mbox{GL}_n(\Z) \right|
 &=&
  g_p(n,i), \\
 \left| H_i \backslash S_{i} \right|
 &=&
  g_p(n-1,i).
\end{eqnarray*}
Furthermore, we have
\begin{eqnarray*}
 \left| H_{i,j} \backslash \mbox{GL}_n(\Z) \right|
 &=&
 p^{i(n-j)} g_p(n,j)\, g_p(j,i), \\
 \left| H_{i,j} \backslash S_{i} \right|
 &=&
 p^{i(n-j)}  g_p(n-1,i)\, g_p(n-i,n-j), \\
 \left| H_{i,j} \backslash S_{i,j} \right|
 &=&
 p^{i(n-1-j)} g_p(n-1,j)\, g_p(j,i).
\end{eqnarray*}
\end{lemma}
\begin{proof}
These are elementary.
We leave details to the reader.
\end{proof}

\begin{lemma}\label{lem:Blambda}
Let $B(\lambda)$ be a function on $\lambda \in \Z^n$.
We put $L_0 := \left(p^2\Z \right)^i \times \left(p\Z \right)^{j-i} \times \Z^{n-j}$.
We assume that the sum
$\displaystyle{
    \sum_{A \in H_{i,j}\backslash GL_n(\Z)} \sum_{\lambda \in L_0} B({^t A} \lambda)}$
converges absolutely.
Then we have
\begin{eqnarray*}
\sum_{A \in H_{i,j}\backslash GL_n(\Z)} \sum_{\lambda \in L_0} B({^t A} \lambda)
    &=& a_0 \sum_{\lambda \in \Z^n} B(\lambda)
   + a_1 \sum_{\lambda \in \Z^n} B(p\lambda) 
   + a_2 \sum_{\lambda \in \Z^n} B(p^2 \lambda),
\end{eqnarray*}
where $a_0$, $a_1$ and $a_2$ are integers which satisfy
\begin{eqnarray*}
 a_0 + a_1 + a_2 = \left| H_{i,j} \backslash  \mbox{GL}_n(\Z) \right|,\
 a_0 + a_1 =\left| H_{i,j} \backslash S_{i} \right|
\mbox{ and }
 a_0 =\left| H_{i,j} \backslash S_{i,j} \right|.
\end{eqnarray*}
\end{lemma}
\begin{proof}
For $\lambda \in \Z^n$ we denote by $\mbox{gcd}(\lambda)$
the greatest common divisor of all entries in $\lambda$.
Let $X$ be a complete set of representatives of $H_{i,j} \backslash \mbox{GL}_n(\Z)$.
For $\lambda \in \Z^n$ we define
\begin{eqnarray*}
    N(\lambda)
  &:=&
    \left|\left\{ A \in X \, | \, \lambda \in {^t A}L_0 \right\} \right|.
\end{eqnarray*}
We remark that $N(\lambda)$ does not depend on the choice of $X$.
To show this lemma, it is enough to calculate $N(\lambda)$ for given $\lambda \in \Z^n$.

By the definition of $S_{i,j}$ and $S_i$, we have
\begin{eqnarray*}
 S_{i,j} &=& \left\{ A \in \mbox{GL}_n(\Z) \, | \, {^t (0,...,0,1)} \in {^t A} L_0 \right\}, \\
 S_i &=& \left\{ A \in \mbox{GL}_n(\Z) \, | \,  {^t (0,...,0,p)} \in {^t A} L_0 \right\}.
\end{eqnarray*}
Hence we have
$N({^t (0,...,0,1)}) = |H_{i,j}\backslash S_{i,j}|$
and
$N({^t (0,...,0,p)}) = |H_{i,j}\backslash S_{i}|$.
Furthermore, we have
$N({^t (0,...,0,p^2)}) = |H_{i,j}\backslash \mbox{GL}_n(\Z)|$.

For any $\lambda \in \Z^n$, there exists a matrix $B \in \mbox{GL}_n(\Z)$
such that ${^t B} \lambda = \mbox{gcd}(\lambda) {^t (0,...,0,1)}$.
Thus we have $N(\lambda) = N(\mbox{gcd}(\lambda) {^t (0,...,0,1)})$.
Hence $N(\lambda)$ equals to
  $\left| H_{i,j} \backslash S_{i,j} \right|$,
  $|H_{i,j}\backslash S_{i}|$ or
  $\left| H_{i,j} \backslash  \mbox{GL}_n(\Z) \right|$,
  according as
  $\mbox{gcd}(p^2, \mbox{gcd}(\lambda)) = 1$,
  $p$ or $p^2$.
Therefore we obtain this lemma.
\end{proof}

\subsection{Calculation of the function $K_{i,j}^{\alpha}$}
For simplicity we define
\begin{eqnarray*}
  G_j^n(m) := G_j^n(m,\lambda),
\end{eqnarray*}
where $\lambda \in \Z^n$ is an vector which satisfy $\lambda \not \equiv 0 \mod p$.
Due to Proposition~\ref{prop:gauss}, the value $G_j^n(m)$
does not depend on the choice of $\lambda$.

\begin{lemma}\label{lem:K_linear}
 If $p^2|m$, then we have
 \begin{eqnarray*}
     K_{i,j}^{\alpha}
  &=&
       p^{-k(2n-i-j)+(n-j)(n-i+1)} G_{\alpha}^{j-i}(0)\\
  &&
      \times
      \left\{ a_0 E_{k,\frac{m}{p^2}}^{(n)}(\tau,p^2 z)
               + a_1 E_{k,m}^{(n)}(\tau,p z) 
               + a_2 E_{k,mp^2}^{(n)}(\tau,z) 
      \right\},
 \end{eqnarray*}
 where
 \begin{eqnarray*}
     a_0 + a_1 + a_2 = \left| H_{i,j} \backslash  \mbox{GL}_n(\Z) \right|,\
     a_0 + a_1 =\left| H_{i,j} \backslash S_{i} \right|
    \mbox{ and }
     a_0 =\left| H_{i,j} \backslash S_{i,j} \right|.
 \end{eqnarray*}
 If $p^2 {\not |} m$, then we have
 \begin{eqnarray*}
      K_{i,j}^{\alpha}
   &=&
      p^{-k(2n-i-j)+(n-j)(n-i+1)}
      \left\{
        (G_{\alpha}^{j-i}(0) - G_{\alpha}^{j-i}(m))
        \left[\Gamma(\delta_{j}) ; \Gamma(\delta_{i,j}) \right]
      \right.\\
   &&
        \times 
        \left\{
          g_p(n-1,j) E_{k,m}^{(n)}(\tau,pz)
          + p^{n-j} g_p(n-1,j-1) E_{k,mp^2}^{(n)}(\tau,z)
        \right\} \\
   &&
      +\
      G_{\alpha}^{j-i}(m)
      \left[\Gamma(\delta_{i}) ; \Gamma(\delta_{i,j}) \right]\\
   && \left.
        \times 
        \left\{
          g_p(n-1,i) E_{k,m}^{(n)}(\tau,pz)
          + p^{n-i} g_p(n-1,i-1) E_{k,mp^2}^{(n)}(\tau,z)
        \right\}
      \right\},
 \end{eqnarray*}
 where $\Gamma(\delta_{i,j})$ and $\Gamma(\delta_{i})$ are groups denoted in Section~\ref{ss:def_Kij}.

 In particular, the function $K_{i,j}^{\alpha}$ is a linear combination of $E_{k,\frac{m}{p^2}}^{(n)}|U(p^2)$,
 $E_{k,m}^{(n)}| U(p)$ and $E_{k,mp^2}^{(n)}$.
\end{lemma}
\begin{proof}
First we assume $p^2|m$.
In this case the sum $G_{\alpha}^{j-i}(m,\lambda')$ equals to $G_{\alpha}^{j-i}(0)$
for any $\lambda' \in \Z^{j-i}$.
Hence
due to Lemma~\ref{lem:kij_yam},
we obtain
\begin{eqnarray*}
 K_{i,j}^{\alpha}
 &=&
   p^{-k(2n-i-j)+(n-j)(n-i+1)} G_{\alpha}^{j-i}(0)
   \sum_{M \in \Gamma(\delta_{i,j}) \backslash \Gamma_n} \\
   && \times 
   \sum_{\lambda \in (p^2 \Z)^i \times (p \Z)^{j-i} \times \Z^{n-j}}
   j\!\left(k,\frac{m}{p^2}; M , \lambda\right)\!(\tau,p^2z).
\end{eqnarray*}
If $\{A_l\}_l$ is a complete set of representatives of $H_{i,j}  \backslash \mbox{GL}_n(\Z)$, then
the set $\left\{\begin{pmatrix} A_l & 0 \\ 0 & {^t A_l}^{-1} \end{pmatrix} \right\}_l$
is a complete set of representatives of $\Gamma(\delta_{i,j})\backslash \Gamma_{n,0}$.
Hence we have
\begin{eqnarray*}
 K_{i,j}^{\alpha}
 &=&
   p^{-k(2n-i-j)+(n-j)(n-i+1)} G_{\alpha}^{j-i}(0)
   \sum_{M \in \Gamma_{n,0} \backslash \Gamma_n} \sum_{A \in H_{i,j}  \backslash GL_n(\Z)} \\
   && \times 
   \sum_{\lambda \in (p^2 \Z)^i \times (p \Z)^{j-i} \times \Z^{n-j}}
   j\!\left(k,\frac{m}{p^2}; M , {^t A} \lambda\right)\!(\tau,p^2z).
\end{eqnarray*}

From Lemma~\ref{lem:Blambda} we obtain
\begin{eqnarray*}
 K_{i,j}^{\alpha}
 &=&
    p^{-k(2n-i-j)+(n-j)(n-i+1)} G_{\alpha}^{j-i}(0)
   \sum_{M \in \Gamma_{n,0} \backslash \Gamma_n}
   \left\{
     a_0 \sum_{\lambda \in \Z^n} j\!\left(k,\frac{m}{p^2}; M , \lambda\right)\!(\tau,p^2z) \right. \\
  && \left.
     + a_1 \sum_{\lambda \in \Z^n} j\!\left(k,\frac{m}{p^2}; M , p \lambda\right)\!(\tau,p^2z)
     + a_2 \sum_{\lambda \in \Z^n} j\!\left(k,\frac{m}{p^2}; M , p^2 \lambda\right)\!(\tau,p^2z)
   \right\}.
\end{eqnarray*}
Due to the two identities
\begin{eqnarray*}
 j\!\left(k,\frac{m}{p^2}; M , p \lambda\right)\!(\tau,p^2z)
 =
 j\!\left(k,m; M ,  \lambda\right)\!(\tau,p z)
\end{eqnarray*}
and
\begin{eqnarray*}
 j\!\left(k,\frac{m}{p^2}; M , p^2 \lambda\right)\!(\tau,p^2z)
 =
 j\!\left(k,m p^2; M , \lambda\right)\!(\tau, z),
\end{eqnarray*}
we have
\begin{eqnarray*}
   K_{i,j}^{\alpha}
 &=&
   p^{-k(2n-i-j)+(n-j)(n-i+1)} G_{\alpha}^{j-i}(0)\\
   &&
   \times
   \left\{ a_0 E_{k,\frac{m}{p^2}}^{(n)}(\tau,p^2 z)
            + a_1 E_{k,m}^{(n)}(\tau,p z) 
            + a_2 E_{k,mp^2}^{(n)}(\tau,z) 
   \right\}.
\end{eqnarray*}
Thus we showed this lemma for the case $p^2|m$.

We now assume $p^2 {\not |} m$.
In this case the sum $G_{\alpha}^{j-i}(m,\lambda')$ equals to $G_{\alpha}^{j-i}(0)$
or $G_{\alpha}^{j-i}(m)$, according as $\lambda' \in p \Z^{j-i}$ or $\lambda' {\not \in} p \Z^{j-i}$.
Thus
due to Lemma~\ref{lem:kij_yam} we have
\begin{eqnarray*}
 K_{i,j}^{\alpha}
 &=&
   p^{-k(2n-i-j)+(n-j)(n-i+1)}
   \left\{
    (G_{\alpha}^{j-i}(0) - G_{\alpha}^{j-i}(m))
   \sum_{M \in \Gamma(\delta_{i,j}) \backslash \Gamma_n}
      \sum_{\lambda \in (p\Z)^j \times \Z^{n-j}}
   \right. \\
   &&
   \left.
   \times
   j\!\left(k, m; M , \lambda\right)\!(\tau,p z)
   +
   G_{\alpha}^{j-i}(m)
   \sum_{M \in \Gamma(\delta_{i,j}) \backslash \Gamma_n} 
   \sum_{\lambda \in (p\Z)^i \times \Z^{n-i}}
   j\!\left(k, m; M , \lambda\right)\!(\tau,p z)
   \right\} .
\end{eqnarray*}
Here we have
\begin{eqnarray*}
 &&
      \sum_{M \in \Gamma(\delta_{i,j}) \backslash \Gamma_n}
      \sum_{\lambda \in (p\Z)^j \times \Z^{n-j}}
      j\!\left(k, m; M , \lambda\right)\!(\tau,p z)  \\
 &=&
     \left[\Gamma(\delta_{j}) ; \Gamma(\delta_{i,j}) \right]
     \sum_{M \in \Gamma(\delta_{j}) \backslash \Gamma_n}
      \sum_{\lambda \in (p\Z)^j \times \Z^{n-j}}
      j\!\left(k, m; M , \lambda\right)\!(\tau,p z) \\
 &=&
     \left[\Gamma(\delta_{j}) ; \Gamma(\delta_{i,j}) \right]
     \sum_{M \in \Gamma_{n,0} \backslash \Gamma_n}
     \sum_{A \in H_j \backslash GL_n(\Z)}
      \sum_{\lambda \in (p\Z)^j \times \Z^{n-j}}
      j\!\left(k, m; M , {^t A} \lambda\right)\!(\tau,p z) \\
 &=&
    \left[\Gamma(\delta_{j}) ; \Gamma(\delta_{i,j}) \right]
    \left\{
      g_p(n-1,j) E_{k,m}^{(n)}(\tau,pz)
      + p^{n-j} g_p(n-1,j-1) E_{k,mp^2}^{(n)}(\tau,z)
    \right\}
\end{eqnarray*}
and
\begin{eqnarray*}
 &&
      \sum_{M \in \Gamma(\delta_{i,j}) \backslash \Gamma_n}
      \sum_{\lambda \in (p\Z)^i \times \Z^{n-i}}
      j\!\left(k, m; M , \lambda\right)\!(\tau,p z)  \\
 &=&
    \left[\Gamma(\delta_{i}) ; \Gamma(\delta_{i,j}) \right]
    \left\{
      g_p(n-1,i) E_{k,m}^{(n)}(\tau,pz)
      + p^{n-i} g_p(n-1,i-1) E_{k,mp^2}^{(n)}(\tau,z)
    \right\}.
\end{eqnarray*}
Hence we showed this lemma also for the case $p^2 {\not |} m$.
\end{proof}

The following proposition has been shown by Yamazaki~\cite[Theorem 3.3]{Ya2}
for the case $m=1$.
We generalize it for any positive-integer $m$.
\begin{prop}\label{prop:E_linear}
 For any natural number $l$ $(0 \leq l \leq n)$, the form
 $E_{k,m}^{(n)} | V_{l,n-l}(p^2)$ is a linear combination of
 $E_{k,\frac{m}{p^2}}^{(n)}|U(p^2)$, $E_{k,m}^{(n)}| U(p)$ and $E_{k,mp^2}^{(n)}$
 over $\C$.
\end{prop}
\begin{proof}
This proposition follows from Lemma~\ref{lem:E_K} and Lemma~\ref{lem:K_linear}.
\end{proof}

\subsection{Relation among Jacobi-Eisenstein series}

Now we shall calculate the coefficients in
the linear combinations in Proposition~\ref{prop:E_linear}.
This calculation can be directly done by using the values of
$G_{j-i}^{\alpha}(m)$ and $g_p(a,b)$.
However, we will here use the Siegel $\Phi$-operators
for simplicity.

We set
 \begin{eqnarray*}
   \begin{pmatrix}
     a_{0,m,p,k}\\
     a_{1,m,p,k}\\
     a_{2,m,p,k}
   \end{pmatrix}
  &:=&
   \begin{cases}
    \begin{pmatrix}
      p^{-2k+2}\\
      p^{-k}(p-1)\\
      1 
    \end{pmatrix}
    &
    \mbox{if } p^2|m,
    \\
    \begin{pmatrix}
      0 \\
      p^{-2k+2}+p^{-k+1}-p^{-k}\\
      1 
    \end{pmatrix}
    &
    \mbox{if } p^2 {\not |}m \mbox{ and } p|m,
    \\
    \begin{pmatrix}
      0 \\
      p^{-2k+2}-p^{-k}\\
      p^{-k+1}+1
    \end{pmatrix}
    &
    \mbox{if } p {\not |}m.
   \end{cases}
 \end{eqnarray*}

\begin{lemma}\label{lem:E_V_one}
For the Jacobi-Eisenstein series $E_{k,m}^{(1)}$ of degree $1$,
we have the identity
\begin{eqnarray*}
 E_{k,m}^{(1)} | \left( V_{0,1}(p^2), V_{1,0}(p^2) \right)
 &=&
 \left(E_{k,\frac{m}{p^2}}^{(1)}|U(p^2), E_{k,m}^{(1)}|U(p), E_{k,mp^2}^{(1)} \right)
 \begin{pmatrix}
  0         & a_{0,m,p,k} \\
  p^{-k} & a_{1,m,p,k} \\
  0         & a_{2,m,p,k}
 \end{pmatrix}.
\end{eqnarray*}

\end{lemma}
\begin{proof}
Since
$\Gamma_1(p^2 1_2) \Gamma_1
  = \Gamma_1 (p^2 1_2)$,
the relation
$E_{k,m}^{(1)}|V_{0,1}(p^2) = p^{-k} E_{k,m}^{(1)}|U(p)$
is obvious.

From Lemma~\ref{lem:E_K} we obtain
\begin{eqnarray*}
 E_{k,m}^{(1)}|V_{1,0}(p^2)
 &=&
 K_{0,0}^0 + K_{0,1}^1 + K_{1,1}^0.
\end{eqnarray*}

Due to Lemma~\ref{lem:value_a} and Lemma~\ref{lem:K_linear}, we have
\begin{eqnarray*}
 K_{0,0}^0 &=&
 \begin{cases}
   p^{-2k+2} E_{k,\frac{m}{p^2}}^{(1)}(\tau,p^2 z) & \mbox{ if } p^2 | m, \\
   p^{-2k+2} E_{k,m}^{(1)}(\tau,pz) & \mbox{ if } p^2 {\not | }m,
 \end{cases} \\
 K_{0,1}^1 &=&
 \begin{cases}
   p^{-k}(p-1) E_{k,m}^{(1)}(\tau,pz) & \mbox{ if } p | m, \\
   p^{-k+1} E_{k,mp^2}^{(1)}(\tau,z) - p^{-k} E_{k,m}^{(1)}(\tau,pz) & \mbox{ if } p {\not |} m,
 \end{cases} \\
 K_{1,1}^0 &=&
   E_{k, m p^2}^{(1)}(\tau,z).
\end{eqnarray*}
Therefore this lemma follows.
\end{proof}

Let $B_{l,l+1}(X_l)$,
$B_{2,n+1}(X_2,...,X_n)$ and $A_{2,n+1}^{p,k}$ be matrices introduced in Section~\ref{ss:Satake_Siegel_Phi}.
We recall
$    A_{2,n+1}^{p,k}
 =
    B_{2,n+1}(p^{2-k},p^{3-k},...,p^{n-k})
$
 and the matrix $A_{2,n+1}^{p,k}$ has the size $2$ times $(n+1)$.

The following proposition has been shown by Yamazaki~\cite[Theorem 4.1]{Ya2}
for the case $m=1$.
We generalize it for any positive-integer $m$.
\begin{prop}\label{prop:EV}
For any Jacobi-Eisenstein series $E_{k,m}^{(n)}$ of degree $n$,
the following identity holds
\[
  E_{k,m}^{(n)} | \left(V_{0,n}(p^2), ..., V_{n,0}(p^2)\right)\\
 =
  \left(E_{k,\frac{m}{p^2}}^{(n)}|U(p^2),
          E_{k, m}^{(n)}|U(p),
          E_{k, m p^2}^{(n)}
  \right)
  \begin{pmatrix}
   0         & a_{0,m,p,k} \\
   p^{-k} & a_{1,m,p,k} \\
   0         & a_{2,m,p,k}
  \end{pmatrix}
  A_{2,n+1}^{p,k}.
\]
\end{prop}
\begin{proof}
Let m be a positive-integer. Let $\Phi$ be the Siegel $\Phi$-operator
for Jacobi forms introduced in Section~\ref{ss:Phi_Jacobi}.
From~Lemma~\ref{lem:Phi_Tauschen} and
from the fact that $\Phi(E_{k,m}^{(n)}) = E_{k,m}^{(n-1)}$,
we have
\begin{eqnarray*}
   \Phi(E_{k,m}^{(n)}|(V_{0,n}(p^2),...,V_{n,0}(p^2)))
  &=&
   E_{k,m}^{(n-1)}|(V_{0,n-1}(p^2),...,V_{n-1,0}(p^2)) B_{n,n+1}(p^{n-k}).
\end{eqnarray*}
Hence
by using Siegel $\Phi$-operator $n-1$ times and by using Lemma~\ref{lem:E_V_one}, we have
\begin{eqnarray*}
  &&
   \Phi^{(n-1)}(E_{k,m}^{(n)}|(V_{0,n}(p^2),...,V_{n,0}(p^2))) \\
  &=&
   E_{k,m}^{(1)}|(V_{0,1}(p^2),V_{1,0}(p^2)) B_{2,n+1}(p^{2-k},p^{3-k},...,p^{n-k}) \\
  &=&
    \left(E_{k,\frac{m}{p^2}}^{(1)}|U(p^2), E_{k,m}^{(1)}|U(p), E_{k,mp^2}^{(1)} \right)
    \begin{pmatrix}
     0         & a_{0,m,p,k} \\
     p^{-k} & a_{1,m,p,k} \\
     0         & a_{2,m,p,k}
    \end{pmatrix}
    B_{2,n+1}(p^{2-k},p^{3-k},...,p^{n-k}).
\end{eqnarray*}

On the other hand,
due to Proposition~\ref{prop:E_linear}, there exists a $3 \times (n+1)$ matrix $B_n^k$
which satisfies
\begin{eqnarray*}
    E_{k,m}^{(n)}|(V_{0,n}(p^2),...,V_{n,0}(p^2))
  &=&
    \left(E_{k,\frac{m}{p^2}}^{(n)}|U(p^2), E_{k,m}^{(n)}|U(p), E_{k,mp^2}^{(n)} \right)
    B_n^k.
\end{eqnarray*}
From this identity we have
\begin{eqnarray*}
   \Phi^{(n-1)}(E_{k,m}^{(n)}|(V_{0,n}(p^2),...,V_{n,0}(p^2))) 
  &=&
   \left(E_{k,\frac{m}{p^2}}^{(1)}|U(p^2), E_{k,m}^{(1)}|U(p), E_{k,mp^2}^{(1)} \right) B_n^k.
\end{eqnarray*}

Because three forms $E_{k,\frac{m}{p^2}}^{(1)}|U(p^2)$, $E_{k,m}^{(1)}|U(p)$
and $E_{k,mp^2}^{(1)}$ are linearly independent (see Lemma~\ref{lemma:linear_independent}),
we obtain
\begin{eqnarray*}
   B_n^k
 &=&
    \begin{pmatrix}
     0         & a_{0,m,p,k} \\
     p^{-k} & a_{1,m,p,k} \\
     0         & a_{2,m,p,k}
    \end{pmatrix}
    B_{2,n+1}(p^{2-k},p^{3-k},...,p^{n-k})
 \ =\
    \begin{pmatrix}
     0         & a_{0,m,p,k} \\
     p^{-k} & a_{1,m,p,k} \\
     0         & a_{2,m,p,k}
    \end{pmatrix}
    A_{2,n+1}^{p,k}.\\ 
\end{eqnarray*}
Thus this proposition follows for any positive-integer $m$.
\end{proof}

\section{Generalized Maass relation for Siegel-Eisenstein series}\label{sec:Fourier-Jacobi}

The purpose of this section is to prove Theorem~\ref{thm:eins}.
Let $e_{k,m}^{(n)}$ be the $m$-th Fourier-Jacobi coefficient
of Siegel-Eisenstein series $E_k^{(n+1)}$,
which is denoted in Section~\ref{ss:results}.

In this section
we write
$\displaystyle{\sum_{d | m}}$ for
$\displaystyle{\sum_{\begin{smallmatrix} d > 0 \\ d | m \end{smallmatrix}}}$,
and
$\displaystyle{\sum_{d^2 | m}}$ for
$\displaystyle{\sum_{\begin{smallmatrix} d > 0 \\ d^2 | m \end{smallmatrix}}}$,
for simplicity.

\subsection{Fourier-Jacobi coefficients}
We define an arithmetic function
\begin{eqnarray*}
 g_k(m)
 &:=&
 \sum_{d^2 | m}
   \mu(d) \, \sigma_{k-1}\left( \frac{m}{d^2} \right),
\end{eqnarray*}
where $\mu(d)$ is the M\"obius function and we put
$\sigma_{k-1}(m) := \displaystyle{\sum_{d|m} d^{k-1}}$
as usual.

\begin{lemma}\label{lem:gkm}
We obtain
\begin{eqnarray*}
  g_k(mp)
 &=&
  \begin{cases}
    \left( p^{k-1} + 1 \right) g_k(m) & \mbox{ if } p {\not |} m, \\
    p^{k-1} g_k(m)                    & \mbox{ if } p | m.
  \end{cases}
\end{eqnarray*}
\end{lemma}
\begin{proof}
The function $g_{k}(m)$ is a multiplicative function, namely
$g_{k}(ml) = g_{k}(m) g_{k}(l)$ if $\mbox{gcd}(m,l)=1$.
Hence we obtain the identity
 $g_k(m)
   =
  m^{k-1}
  \displaystyle{\prod_{\begin{smallmatrix} q\, :\, prime \\ q|m \end{smallmatrix}}}
  \left( 1 + \frac{1}{q^{k-1}}\right)$.
This lemma follows from this identity.
\end{proof}

The following proposition is a special case of a result in~\cite[Satz 7]{Bo}.
\begin{prop}[Boecherer~\cite{Bo}]\label{prop:Bo_Satz}
We have
\begin{eqnarray*}
 e_{k,m}^{(n)}
 &=&
 \sum_{d^2 | m }
   g_k\!\left(\frac{m}{d^2}\right)\, E_{k,\frac{m}{d^2}}^{(n)}|U(d).
\end{eqnarray*}
\end{prop}
\begin{proof}
For the proof of this proposition,
the reader is referred to~\cite[Theorem 5.5]{Ya}.
\end{proof}

\begin{prop}\label{prop:sum_EU}
For any $n > 0$ and for any $m>0$
we have the identity
\begin{eqnarray*}
 &&
\sum_{ d^2 | m }
  g_k\! \left(\frac{m}{d^2}\right) \,
  \left(E_{k,\frac{m}{p^2 d^2}}^{(n)}|U(p^2 d),
          E_{k,\frac{m}{d^2}}^{(n)}|U(p d),
          E_{k,\frac{m p^2}{d^2}}^{(n)}|U(d)
  \right)\!\!
  \begin{pmatrix}
   a_{0,\frac{m}{d^2},p,k} \\
   a_{1,\frac{m}{d^2},p,k} \\
   a_{2,\frac{m}{d^2},p,k}
  \end{pmatrix}
  \\
  &=&
\left(
  e_{k,\frac{m}{p^2}}^{(n)}|U(p^2), 
  e_{k,m}^{(n)}|U(p),
  e_{k,mp^2}^{(n)}
\right)\!\!
\begin{pmatrix}
  1 \\
  p^{-k}(-1+p\, \delta_{p|m})  \\
  p^{-2k+2}
\end{pmatrix},
\end{eqnarray*}
where $\delta_{p|m}$ is $1$ or $0$, according as $p|m$ or $p {\not |} m$.
\end{prop}
\begin{proof}
Due to Proposition~\ref{prop:Bo_Satz} and Lemma~\ref{lem:gkm},
this proposition is obtained by straightforward calculation as follows.
We set nine functions
\begin{eqnarray*}
   Eg_1
  &:=&
   \sum_{\begin{smallmatrix} d^2|m \\ \frac{m}{d^2} \equiv 0\, (p^2)\end{smallmatrix}}
       p^{-2k+2} E_{k,\frac{m}{p^2d^2}}^{(n)}|U(p^2d)\, g\!\left(\frac{m}{d^2}\right),\\
   Eg_2
  &:=&
   \sum_{\begin{smallmatrix} d^2|m \\ \frac{m}{d^2} \equiv 0\, (p^2)\end{smallmatrix}}
       p^{-k}(p-1) E_{k,\frac{m}{d^2}}^{(n)}|U(pd)\, g\!\left(\frac{m}{d^2}\right),\\
   Eg_3
  &:=&
   \sum_{\begin{smallmatrix} d^2|m \\ \frac{m}{d^2} \equiv 0\, (p^2)\end{smallmatrix}}
       E_{k,\frac{p^2m}{d^2}}^{(n)}|U(d)\, g\!\left(\frac{m}{d^2}\right),
\end{eqnarray*}
\begin{eqnarray*}
   Eg_4
  &:=&
   \sum_{\begin{smallmatrix} d^2|m \\ \frac{m}{d^2} \equiv 0\, (p) \\ \frac{m}{d^2} \not \equiv 0\, (p^2)\end{smallmatrix}}
       p^{-2k+2} E_{k,\frac{m}{d^2}}^{(n)}|U(pd) \, g\!\left(\frac{m}{d^2}\right),\\
   Eg_5
  &:=&
   \sum_{\begin{smallmatrix} d^2|m \\ \frac{m}{d^2} \equiv 0\, (p) \\ \frac{m}{d^2} \not \equiv 0\, (p^2)\end{smallmatrix}}
       (p^{-k+1}-p^{-k}) E_{k,\frac{m}{d^2}}^{(n)}|U(pd)\, g\!\left(\frac{m}{d^2}\right),\\
   Eg_6
  &:=&
   \sum_{\begin{smallmatrix} d^2|m \\ \frac{m}{d^2} \equiv 0\, (p) \\ \frac{m}{d^2} \not \equiv 0\, (p^2)\end{smallmatrix}}
       E_{k,\frac{p^2m}{d^2}}^{(n)}|U(d)\, g\!\left(\frac{m}{d^2}\right),
\end{eqnarray*}
\begin{eqnarray*}
   Eg_7
  &:=&
   \sum_{\begin{smallmatrix} d^2|m \\ \frac{m}{d^2} \not \equiv 0\, (p) \end{smallmatrix}}
       p^{-2k+2} E_{k,\frac{m}{d^2}}^{(n)}|U(pd) \, g\!\left(\frac{m}{d^2}\right),\\
   Eg_8
  &:=&
   \sum_{\begin{smallmatrix} d^2|m \\ \frac{m}{d^2} \not \equiv 0\, (p) \end{smallmatrix}}
       (-p^{-k}) E_{k,\frac{m}{d^2}}^{(n)}|U(pd) \, g\!\left(\frac{m}{d^2}\right),
\end{eqnarray*}
and
\begin{eqnarray*}
   Eg_9
  &:=&
   \sum_{\begin{smallmatrix} d^2|m \\ \frac{m}{d^2} \not \equiv 0\, (p) \end{smallmatrix}}
       (p^{-k+1}+1) E_{k,\frac{p^2m}{d^2}}^{(n)}|U(d) \, g\!\left(\frac{m}{d^2}\right).
\end{eqnarray*}
If $\mbox{ord}_p m \equiv 1 \, (2)$, then
\begin{eqnarray*}
 &&
  \sum_{d^2 | m }
    g_k\! \left(\frac{m}{d^2}\right) \,
    \left(E_{k,\frac{m}{p^2 d^2}}^{(n)}|U(p^2 d),
            E_{k,\frac{m}{d^2}}^{(n)}|U(p d),
            E_{k,\frac{m p^2}{d^2}}^{(n)}|U(d)
    \right)\!\!
    \begin{pmatrix}
     a_{0,\frac{m}{d^2},p,k} \\
     a_{1,\frac{m}{d^2},p,k} \\
     a_{2,\frac{m}{d^2},p,k}
    \end{pmatrix}\\
    &=&
    Eg_1 + Eg_2 + Eg_3 + Eg_4 + Eg_5 + Eg_6,
\end{eqnarray*}
and
\begin{eqnarray*}
  Eg_1 &=& e_{k,\frac{m}{p^2}}^{(n)}|U(p^2), \\
  Eg_2 + Eg_5 &=& p^{-k}(p-1) e_{k,m}^{(n)}|U(p),\\
  Eg_3 + Eg_4 + Eg_6 &=& p^{-2k+2} e_{k,mp^2}^{(n)}.
\end{eqnarray*}
Because of the assumption $\mbox{ord}_p m \equiv 1\, (2)$, we have $\delta_{p|m} = 1$.

Hence this proposition follows for the case $\mbox{ord}_p m \equiv 1 \, (2)$.

If $\mbox{ord}_p m \equiv 0 \, (2)$, then
\begin{eqnarray*}
 &&
  \sum_{d^2 | m }
    g_k\! \left(\frac{m}{d^2}\right) \,
    \left(E_{k,\frac{m}{p^2 d^2}}^{(n)}|U(p^2 d),
            E_{k,\frac{m}{d^2}}^{(n)}|U(p d),
            E_{k,\frac{m p^2}{d^2}}^{(n)}|U(d)
    \right)\!\!
    \begin{pmatrix}
     a_{0,\frac{m}{d^2},p,k} \\
     a_{1,\frac{m}{d^2},p,k} \\
     a_{2,\frac{m}{d^2},p,k}
    \end{pmatrix}\\
    &=&
    Eg_1 + Eg_2 + Eg_3 + Eg_7 + Eg_8 + Eg_9,
\end{eqnarray*}
and
\begin{eqnarray*}
   Eg_1
  &=&
   \delta_{p^2|m}\left\{ e_{k,\frac{m}{p^2}}^{(n)}|U(p^2)
   + p^{-k+1} \sum_{\begin{smallmatrix} d^2|m \\ \frac{m}{d^2} \not \equiv 0\, (p^2) \end{smallmatrix}}
     E_{k,\frac{m}{d^2}}^{(n)}|U(pd)\, g\!\left(\frac{m}{d^2}\right)\right\},\\
   Eg_2+Eg_8
  &=&
   \delta_{p^2|m} \left\{p^{-k+1} e_{k,m}^{(n)}|U(p)
    - p^{-k+1} \sum_{\begin{smallmatrix} d^2|m \\ \frac{m}{d^2} \not \equiv 0\, (p^2) \end{smallmatrix}}
     E_{k,\frac{m}{d^2}}^{(n)}|U(pd)\, g\!\left(\frac{m}{d^2}\right) \right\} \\
  &&
    -p^{-k} e_{k,m}^{(n)}|U(p),\\
   Eg_3 + Eg_7 + Eg_9
  &=&
   p^{-2k+2} e_{k,mp^2}^{(n)}.
\end{eqnarray*}
Here $\delta_{p^2|m}$ is defined by $1$ or $0$, according as $p^2|m$ or $p^2 {\not |} m$.
Because of the assumption $\mbox{ord}_p m \equiv 0\, (2)$, we have $\delta_{p^2|m} = \delta_{p|m}$.

Therefore this proposition follows also for the case $\mbox{ord}_p m \equiv 0 \, (2)$.
\end{proof}

\subsection{Proof of Theorem~\ref{thm:eins}}
Now we shall prove Theorem~\ref{thm:eins}.
For any prime $p$ and for any positive-integer $d$, the operators $V_{l,n-l}(p^2)$ and $U(d)$ are compatible.
Hence from Proposition~\ref{prop:Bo_Satz}, Proposition~\ref{prop:EV} and Proposition~\ref{prop:sum_EU},
we have
\begin{eqnarray*}
  &&
    e_{k,m}^{(n)}|\left(V_{0,n}(p^2), ..., V_{n,0}(p^2)\right) \\
  &=&
    \sum_{ d^2 | m }
       g_k\!\left(\frac{m}{d^2}\right)\,
          \left(E_{k,\frac{m}{d^2}}^{(n)}|\left(V_{0,n}(p^2), ..., V_{n,0}(p^2)\right)|U(d)\right)\\
  &=&
    \sum_{d^2 | m }
       g_k\!\left(\frac{m}{d^2}\right)\,
      \left(E_{k,\frac{m}{p^2 d^2}}^{(n)}|U(p^2 d),
          E_{k, \frac{m}{d^2}}^{(n)}|U(p d),
          E_{k, \frac{m p^2}{d^2}}^{(n)}|U(d)
      \right)
      \begin{pmatrix}
       0         & a_{0,\frac{m}{d^2},p,k} \\
       p^{-k} & a_{1,\frac{m}{d^2},p,k} \\
       0         & a_{2,\frac{m}{d^2},p,k}
      \end{pmatrix}
      A_{2,n+1}^{p,k}\\
  &=&
     \left(
      e_{k,\frac{m}{p^2}}^{(n)}|U(p^2), 
      e_{k,m}^{(n)}|U(p),
      e_{k,mp^2}^{(n)}
    \right)\!\!
    \begin{pmatrix}
      0      & 1 \\
      p^{-k} & p^{-k}(-1+p\, \delta_{p | m})  \\
      0      & p^{-2k+2}
    \end{pmatrix}
    A_{2,n+1}^{p,k}.
\end{eqnarray*}
Thus we obtain Theorem~\ref{thm:eins}.

\section{Generalized Maass relation for the Miyawaki-Ikeda lifts}\label{sec:Miyawaki-Ikeda}

In this section we shall show Theorem~\ref{thm:phi_maass}, Theorem~\ref{thm:drei}
and Corollary~\ref{coro:vier}.
Let $\phi_m \in J_{k+n,m}^{(2n-1)}$ be the $m$-th Fourier-Jacobi coefficient
of the Duke-Imamoglu-Ibukiyama-Ikeda lift $F$ stated in Theorem~\ref{thm:phi_maass}.

In this section the letters $p$ and $q$ are reserved for prime numbers.
For example, the symbol $\prod_{p|N}$ denotes the product over primes $p$ such that $p|N$.

\subsection{Fourier coefficients of $\phi_m$}\label{ss:fourier_phi}

We take the Fourier expansion of $\phi_m$:
\begin{eqnarray*}
 \phi_m(\tau,z)
 &=&
   \sum_{N,R} C_m(N,R) e(N\tau) e({^t R} z),
\end{eqnarray*}
where in the summation $N \in \mbox{Sym}_{2n-1}^*$ and $R \in \Z^{2n-1}$
run over all elements which satisfy $4 Nm - R {^t R} > 0$.
We set
$M = \begin{pmatrix} N& \frac12 R \\ \frac12 {^t R}& m \end{pmatrix}$.
We denote by $D_M$ and by $f_M$ the integers such that
$(-1)^n \det(2 M) = D_M f_M^2$, where $D_M$ is a fundamental discriminant
and $f_M$ is a positive integer.
Then the $(N,R)$-th Fourier coefficient $C_m(N,R)$ of $\phi_m$ is
\begin{eqnarray*}
    C_m(N,R)
  &=&
    C(|D_M|) f_{M}^{k-\frac12} \prod_{p|f_M} \widetilde{F}_p(M;\alpha_p),
\end{eqnarray*}
where $C(|D_M|)$ is the $|D_M|$-th Fourier coefficient of $h$ which corresponds
to $g$ by Shimura correspondence, and
$\widetilde{F}_p(M;X_p) \in \C[X_p + X_p^{-1}]$ is a certain Laurent polynomial
introduced in~\cite[\S 1]{Ik}.

\subsection{Matrix $A_{2,2n}^{p,k+n}$}

Let $A_{2,n+1}^{p,k}$ and $A'_{2,2n}(X_p)$ be the matrices introduced in Section~\ref{ss:Satake_Siegel_Phi}.
 
\begin{lemma}\label{lem:Adash}
For any even integer $k$ we obtain
\begin{eqnarray*}
    A_{2,2n}^{p,k+n}
  &=&
    p^{-(n-1)(2k-1)} A'_{2,2n}(p^{-(k-\frac12)}).
\end{eqnarray*}
\end{lemma}
\begin{proof}
From the definition of $A_{2,n+1}^{p,k}$ we get
\begin{eqnarray*}
  A_{2,2n}^{p,k+n}
 &=&
  B_{2,2n}(p^{2-n-k},p^{3-n-k},...,p^{n-1-k})\\
 &=&
  \left(\prod_{i=2}^{2n-1}p^{i-n-k}\right)
   B'_{2,2n}(p^{\frac32-n-(k-\frac12)},p^{\frac52-n-(k-\frac12)},...,p^{-\frac32+n-(k-\frac12)})\\
 &=&
  p^{-(n-1)(2k-1)}\, A'_{2,2n}(p^{-(k-\frac12)}).
\end{eqnarray*}
\end{proof}

\subsection{Proof of Theorem~\ref{thm:phi_maass}}\label{ss:pr_of_thm2}

Let $g \in \mbox{GSp}_{2n-1}^+(\R) \cap M_{4n-2}(\Z)$ be a matrix
such that the similitude of $g$ is $\nu(g) = p^2$.
We write the coset decomposition
$\Gamma_{2n-1} g \Gamma_{2n-1} = \displaystyle{\bigcup_i \Gamma_n g_i}$
with $g_i = \begin{pmatrix} A_i & B_i \\ 0_{2n-1} & D_i \end{pmatrix}$.
We take the Fourier expansion of $\phi_m|V(\Gamma_{2n-1}g\Gamma_{2n-1})$:
\begin{eqnarray*}
   \left(\phi_m|V(\Gamma_{2n-1}g\Gamma_{2n-1})\right)(\tau,z)
 &=&
   \sum_{N,R} C_m(g;N,R)\, e(N\tau)\, e({^t R} z),
\end{eqnarray*}
where in the summation $N \in \mbox{Sym}_{2n-1}^*$ and $R \in \Z^{2n-1}$
run over all elements which satisfy $4Nmp^2 - R {^t R} > 0$.

We now fix $N \in \mbox{Sym}_{2n-1}^*$ and $R \in \Z^{2n-1}$ such that $4 Nmp^2 - R {^t R} > 0$.
And we set $M_1 = \begin{pmatrix} N & \frac{1}{2p} R \\ \frac{1}{2p} {^t R} & m \end{pmatrix}$.
\begin{lemma}\label{lem:V_Fourier}
The $(N,R)$-th Fourier coefficient $C_m(g;N,R)$ of $\phi_m|V(\Gamma_{2n-1}g\Gamma_{2n-1})$ is
\begin{eqnarray*}
   C_m(g;N,R)
  &=&
    p^{-(2n-1)(k-\frac12)} C\left(|D_{M_1}| \right) f_{M_1}^{k-\frac12}
    \sum_i \det D_i^{-n-\frac12}\\
  && \times\!\!
    \prod_{q|f_{M_1[diag(p^{-1} {^t D_i},1)]}}\!\!
     \widetilde{F}_q\left(M_1[\mbox{diag}(p^{-1} {^t D_i} ,1)];\alpha_q \right).
\end{eqnarray*}
Here we regard $\widetilde{F}_q\left(M_1[\mbox{diag}(p^{-1} {^t D_i} ,1)];X_q \right)$ as $0$,
if $M_1[\mbox{diag}(p^{-1} {^t D_i} ,1)] \not \in \mbox{Sym}_{2n}^*$.
\end{lemma}
\begin{proof}
From the definition of $V(\Gamma_{2n-1} g \Gamma_{2n-1})$
the $(N,R)$-th Fourier coefficient of the form $\phi_m|V(\Gamma_{2n-1}g\Gamma_{2n-1})$ is
\begin{eqnarray*}
   \sum_i \det D_i^{-(k+n)} C\left(|D_{M_{1,i}}| \right)
   f_{M_{1,i}}^{k-\frac12}
   \prod_{q|f_{M_{1,i}}}\!\!
   \widetilde{F}_q\left(M_{1,i};\alpha_q \right),
\end{eqnarray*}
where $M_{1,i} := M_1[diag(p^{-1} {^t D_i},1)]$.
Thus this lemma follows from the fact that if $M_{1,i}$
is a half-integral symmetric matrix, then
$D_{M_{1,i}} = D_{M_1}$ and
$f_{M_{1,i}} = p^{-(2n-1)} (\det D_i) f_{M_1}$.
\end{proof}
Now we shall prove Theorem~\ref{thm:phi_maass}.
In the same manner as in Lemma~\ref{lem:V_Fourier} we obtain the fact that the $(N,R)$-th Fourier coefficient of
$e_{k+n,m}^{(2n-1)}|V(\Gamma_{2n-1}g\Gamma_{2n-1})$ is
\begin{eqnarray*}
  &&
    p^{-(2n-1)(k-\frac12)} h_{k+\frac12}\left(|D_{M_1}| \right) f_{M_1}^{k-\frac12}
    \sum_i \det D_i^{-n-\frac12}\\
  && \times\!\!
    \prod_{q|f_{M_1[diag(p^{-1} {^t D_i},1)]}}\!\!
     \widetilde{F}_q\left(M_1[\mbox{diag}(p^{-1} {^t D_i} ,1)]; q^{k-\frac12} \right),
\end{eqnarray*}
where $h_{k+\frac12}(|D_{M_1}|)$ is the $|D_{M_1}|$-th Fourier coefficient
of the Cohen type Eisenstein series of weight $k+\frac12$ which corresponds
to the Eisenstein series of weight $2k$ by the Shimura correspondence.

By the virtue of Theorem~\ref{thm:eins}, the form
$e_{k+n,m}^{(2n-1)}|V(\Gamma_{2n-1}g\Gamma_{2n-1})$
is a linear combination of
  $e_{k+n,\frac{m}{p^2}}^{(2n-1)}|U(p^2)$,
  $e_{k+n,m}^{(2n-1)}|U(p)$
  and $e_{k+n,mp^2}^{(2n-1)}$.
Hence there exists constants $u_0$, $u_1$ and $u_2$, such that
\begin{eqnarray*}
    e_{k+n,m}^{(2n-1)}|V(\Gamma_{2n-1}g\Gamma_{2n-1})
  &=&
      u_0\, e_{k+n,\frac{m}{p^2}}^{(2n-1)}|U(p^2)
    + u_1\, e_{k+n,m}^{(2n-1)}|U(p)
    + u_2\, e_{k+n,mp^2}^{(2n-1)}.
\end{eqnarray*}
We remark that the constants $u_0$, $u_1$ and $u_2$ depend
on the choices of $p$, $k$, $m$ and $n$.
The $(N,R)$-th Fourier coefficient of the form of the above RHS is
\begin{eqnarray*}
   &&
      u_0\,    h_{k+\frac12}\left(|D_{M_1}| \right) p^{-k+\frac12} f_{M_1}^{k-\frac12}
      \prod_{q|f_{M_0}}\!\!
       \widetilde{F}_q\left(M_0; q^{k-\frac12} \right) \\
   && +
      u_1\,    h_{k+\frac12}\left(|D_{M_1}| \right) f_{M_1}^{k-\frac12}
      \prod_{q|f_{M_1}}\!\!
       \widetilde{F}_q\left(M_1; q^{k-\frac12} \right) \\
   && +
      u_2\,    h_{k+\frac12}\left(|D_{M_1}| \right) p^{k-\frac12} f_{M_1}^{k-\frac12}
      \prod_{q|f_{M_2}}\!\!
       \widetilde{F}_q\left(M_2; q^{k-\frac12} \right),
\end{eqnarray*}
where
  $M_0 = \begin{pmatrix} N&\frac{1}{2p^2}R \\ \frac{1}{2p^2} {^t R}& \frac{m}{p^2} \end{pmatrix}$
and
  $M_2 = \begin{pmatrix} N&\frac{1}{2}R \\ \frac{1}{2} {^t R}& m p^2 \end{pmatrix}$.
Because $h_{k+\frac12}\left(|D_{M_1}| \right) \neq 0$, we obtain
\begin{eqnarray}
  && \label{eq:V_Fourier}
  \\
  \nonumber
  &&
    p^{-(2n-1)(k-\frac12)} \sum_i \det D_i^{-n-\frac12}
    \prod_{q|f_{M_1[diag(p^{-1} {^t D_i},1)]}}\!\!
     \widetilde{F}_q\left(M_1[\mbox{diag}(p^{-1} {^t D_i} ,1)]; q^{k-\frac12} \right)\\
  \nonumber
  &=&
      u_0\, p^{-k+\frac12}
      \prod_{q|f_{M_0}}\!\!
       \widetilde{F}_q\left(M_0; q^{k-\frac12} \right) 
      +
      u_1\, 
      \prod_{q|f_{M_1}}\!\!
       \widetilde{F}_q\left(M_1; q^{k-\frac12} \right)
    +
      u_2\, p^{k-\frac12}
      \prod_{q|f_{M_2}}\!\!
       \widetilde{F}_q\left(M_2; q^{k-\frac12} \right).
\end{eqnarray}
We denote by $c_0(N,R)$, $c_1(N,R)$ and $c_2(N,R)$ the $(N,R)$-th Fourier coefficients
of $e_{k+n,\frac{m}{p^2}}^{(2n-1)}|U(p^2)$, $e_{k+n,m}^{(2n-1)}|U(p)$ and $e_{k+n, m p^2}^{(2n-1)}$, respectively.
We remark $c_0(N,R) = 0$ if $p^2 {\not |} m$.
Furthermore, we remark that $c_0(N,R) = 0$ if $R \notin p^2 \Z^{2n-1}$, and $c_1(N,R) = 0$ if $R \notin p\Z^{2n-1}$.

Because the forms in the set $\left\{E_{\frac{m}{d^2},k}^{(2n-1)}|U(d)\right\}_{d}$,
where $d$ runs over all positive-integers such that $d^2|m$,
are linearly independent (see Lemma~\ref{lemma:linear_independent})
and because of Proposition~\ref{prop:Bo_Satz},
%
%
three forms $e_{k+n,\frac{m}{p^2}}^{(2n-1)}|U(p^2)$, $e_{k+n,m}^{(2n-1)}|U(p)$ and $e_{k+n, m p^2}^{(2n-1)}$
are linearly independent.

From now on we assume $p^2|m$ for simplicity.
The proof of Theorem~\ref{thm:phi_maass} for the case $p^2 {\not |} m$ is similar
to the case $p^2 | m$.

There exist pairs of matrices $(N_j,R_j)$ $(j=1,2,3)$
such that
\begin{eqnarray*}
  \det\left(
   \begin{pmatrix}
     c_0(N_1,R_1) & c_1(N_1,R_1) & c_2(N_1,R_1) \\
     c_0(N_2,R_2) & c_1(N_2,R_2) & c_2(N_2,R_2) \\
     c_0(N_3,R_3) & c_1(N_3,R_3) & c_2(N_3,R_3)
   \end{pmatrix}
  \right)
  &\neq&
  0.
\end{eqnarray*}
For $j = 1,2,3$, we define
\begin{eqnarray*}
      M_0^{(j)}
    :=
      \begin{pmatrix}
        N_j & \dfrac{1}{2p^2} R_j \\
        \dfrac{1}{2p^2} R_j & \dfrac{m}{p^2}
      \end{pmatrix},\
      M_1^{(j)}
    :=
      \begin{pmatrix}
        N_j & \dfrac{1}{2p} R_j \\
        \dfrac{1}{2p} R_j & m
      \end{pmatrix}, \
        M_2^{(j)}
    :=
      \begin{pmatrix}
        N_j & \dfrac12 R_j \\
        \dfrac12 R_j & m p^2
      \end{pmatrix},
\end{eqnarray*}
and we put a $3\times 3$ matrix
\begin{eqnarray*}
     C\!\left({\left\{(N_j,R_j)\right\}}_j;\, \{X_q\}_{q:prime}\right)
     &:=&
     \left(
        \prod_{q|f_{M_i^{(j)}}}\!\!
         \widetilde{F}_q\left(M_i^{(j)}; X_q \right)
     \right)_{\begin{smallmatrix} j = 1,2,3 \\ i = 0,1,2 \end{smallmatrix}}.
\end{eqnarray*}
Then from the identity (\ref{eq:V_Fourier}) we have
\begin{eqnarray*}
    &&
      p^{-(2n-1)(k-\frac12)} 
      \sum_i \left(\det D_i\right)^{-n-\frac12}\\
    &&
     \times
      \begin{pmatrix}
        \prod_{q|f_{M_1^{(1)}[diag(p^{-1} {^t D_i},1)]}}\!\!
         \widetilde{F}_q\left(M_1^{(1)}[\mbox{diag}(p^{-1} {^t D_i} ,1)]; q^{k-\frac12} \right)\\
         \prod_{q|f_{M_1^{(2)}[diag(p^{-1} {^t D_i},1)]}}\!\!
         \widetilde{F}_q\left(M_1^{(2)}[\mbox{diag}(p^{-1} {^t D_i} ,1)]; q^{k-\frac12} \right)\\
         \prod_{q|f_{M_1^{(3)}[diag(p^{-1} {^t D_i},1)]}}\!\!
         \widetilde{F}_q\left(M_1^{(3)}[\mbox{diag}(p^{-1} {^t D_i} ,1)]; q^{k-\frac12} \right)
       \end{pmatrix}\\
    &=&
        C\!\left({\left\{(N_j,R_j)\right\}}_j;\, \{q^{k-\frac12}\}_{q}\right)
       \,
        \begin{pmatrix}
          u_0\, p^{k-\frac12}\\
          u_1\\
          u_2\, p^{-k+\frac12}
        \end{pmatrix}.
\end{eqnarray*}
Hence we obtain
\begin{eqnarray*}
    &&
      \sum_i \left(\det D_i\right)^{-n-\frac12}\, C\!\left({\left\{(N_j,R_j)\right\}}_j;\, \{q^{k-\frac12}\}_{q}\right)^{-1}\\
    &&
     \times      
      \begin{pmatrix}
        \prod_{q|f_{M_1^{(1)}[diag(p^{-1} {^t D_i},1)]}}\!\!
         \widetilde{F}_q\left(M_1^{(1)}[\mbox{diag}(p^{-1} {^t D_i} ,1)]; q^{k-\frac12} \right)\\
         \prod_{q|f_{M_1^{(2)}[diag(p^{-1} {^t D_i},1)]}}\!\!
         \widetilde{F}_q\left(M_1^{(2)}[\mbox{diag}(p^{-1} {^t D_i} ,1)]; q^{k-\frac12} \right)\\
         \prod_{q|f_{M_1^{(3)}[diag(p^{-1} {^t D_i},1)]}}\!\!
         \widetilde{F}_q\left(M_1^{(3)}[\mbox{diag}(p^{-1} {^t D_i} ,1)]; q^{k-\frac12} \right)
       \end{pmatrix}\\
    &=&
        p^{(2n-1)(k-\frac12)} 
        \begin{pmatrix}
          u_0\, p^{k-\frac12}\\
          u_1\\
          u_2\, p^{-k+\frac12}
        \end{pmatrix}.
\end{eqnarray*}
The RHS of the above identity does not depend on the choices of
$(N_j,R_j)$ $(j=1,2,3)$.
Furthermore, the above identity holds for infinitely many integer $k$.
Therefore there exist Laurent polynomials $\Phi_i(X_p)$ $\in$ $\C[X_p+X_p^{-1}]$ $(i=0,1,2)$
which are independent of the choices of $(N_j,R_j)$ $(j=1,2,3)$,
such that
\begin{eqnarray*}
    &&
      \sum_i \left(\det D_i\right)^{-n-\frac12}\, C\!\left( {\left\{(N_j,R_j)\right\}}_j;\, \{X_q\}_{q}\right)^{-1}\\
    &&
     \times      
      \begin{pmatrix}
        \prod_{q|f_{M_1^{(1)}[diag(p^{-1} {^t D_i},1)]}}\!\!
         \widetilde{F}_q\left(M_1^{(1)}[\mbox{diag}(p^{-1} {^t D_i} ,1)]; X_q \right)\\
         \prod_{q|f_{M_1^{(2)}[diag(p^{-1} {^t D_i},1)]}}\!\!
         \widetilde{F}_q\left(M_1^{(2)}[\mbox{diag}(p^{-1} {^t D_i} ,1)]; X_q \right)\\
         \prod_{q|f_{M_1^{(3)}[diag(p^{-1} {^t D_i},1)]}}\!\!
         \widetilde{F}_q\left(M_1^{(3)}[\mbox{diag}(p^{-1} {^t D_i} ,1)]; X_q \right)
       \end{pmatrix}\\
    &=&
    \begin{pmatrix}
      \Phi_0(X_p)\\
      \Phi_1(X_p)\\
      \Phi_2(X_p)
    \end{pmatrix}.
\end{eqnarray*}
In particular, we have
\begin{eqnarray*}
  &&
    \sum_i \det D_i^{-n-\frac12}
    \prod_{q|f_{M_1[diag(p^{-1} {^t D_i},1)]}}\!\!
     \widetilde{F}_q\left(M_1[\mbox{diag}(p^{-1} {^t D_i} ,1)]; X_q \right)\\
  &=&
      \Phi_0(X_p)
      \prod_{q|f_{M_0}}\!\!
       \widetilde{F}_q\left(M_0; X_q \right) 
      +
      \Phi_1(X_p)\,
      \prod_{q|f_{M_1}}\!\!
       \widetilde{F}_q\left(M_1; X_q \right)
    +
      \Phi_2(X_p)\, 
      \prod_{q|f_{M_2}}\!\!
       \widetilde{F}_q\left(M_2; X_q \right).    
\end{eqnarray*}
Therefore,
by substituting $X_q = \alpha_q$ in the above identity and
by using the relations
$p f_{M_0} = f_{M_1} = p^{-1} f_{M_2}$ and
$D_{M_0} = D_{M_1} = D_{M_2}$,
we obtain
\begin{eqnarray*}
  &&
    p^{-(2n-1)(k-\frac12)} C\left(|D_{M_1}| \right) f_{M_1}^{k-\frac12}
    \sum_i \det D_i^{-n-\frac12}\\
  && \times\!\!
    \prod_{q|f_{M_1[diag(p^{-1} {^t D_i},1)]}}\!\!
     \widetilde{F}_q\left(M_1[\mbox{diag}(p^{-1} {^t D_i} ,1)];\alpha_q \right) \\
  &=&
    p^{-(2n-1)(k-\frac12)}
   \left\{
    \Phi_0(\alpha_p)\,  p^{k-\frac12} C\left(|D_{M_0}| \right) f_{M_0}^{k-\frac12}
      \prod_{q|f_{M_0}}\!\!
       \widetilde{F}_q\left(M_0; \alpha_q \right) \right. \\
  &&\ \qquad +
   \left. \Phi_1(\alpha_p)\,  C\left(|D_{M_1}| \right) f_{M_1}^{k-\frac12}
      \prod_{q|f_{M_1}}\!\!
       \widetilde{F}_q\left(M_1; \alpha_q \right) \right. \\
  &&\ \qquad +
   \left. \Phi_2(\alpha_p)\,  p^{-k+\frac12} C\left(|D_{M_2}| \right) f_{M_2}^{k-\frac12}
      \prod_{q|f_{M_2}}\!\!
       \widetilde{F}_q\left(M_2; \alpha_q \right) \right\} .
\end{eqnarray*}
Thus
\[
   \phi_m|V(\Gamma_{2n-1} g \Gamma_{2n-1})\\
 =
   p^{-(2n-1)(k-\frac12)}
   \left(p^{k-\frac12} \phi_{\frac{m}{p^2}}|U(p^2), \phi_{m}|U(p), p^{-k+\frac12} \phi_{mp^2} \right)
   \begin{pmatrix}
     \Phi_0(\alpha_p)\\
     \Phi_1(\alpha_p)\\
     \Phi_2(\alpha_p)
   \end{pmatrix}.
\]

Hence
there exist Laurent polynomials $\Phi_{j,l}(X_p) \in \C[X_p+X_p^{-1}]$  $(j=0,1,2,\ l=0,...,2n-1)$ which satisfy
\begin{eqnarray}
  \label{id:phi_m_V}
  &&\\
  \nonumber
  &&
    \phi_m|(V_{0,2n-1}(p^2),...,V_{2n-1,0}(p^2))\\
  &=&
   p^{-(2n-1)(k-\frac12)}
   \left(p^{k-\frac12} \phi_{\frac{m}{p^2}}|U(p^2), \phi_{m}|U(p), p^{-k+\frac12} \phi_{mp^2} \right)
   \begin{pmatrix}
     \Phi_{0,0}(\alpha_p) & \cdots & \Phi_{0,2n-1}(\alpha_p)\\
     \Phi_{1,0}(\alpha_p) & \cdots & \Phi_{1,2n-1}(\alpha_p)\\
     \Phi_{2,0}(\alpha_p) & \cdots & \Phi_{2,2n-1}(\alpha_p)
   \end{pmatrix}.
   \nonumber
\end{eqnarray}
Here the polynomials $\Phi_{j,l}(X_p) $
depend on the choices of $p$ and $m$,
but not on the choice of $f$ which is the preimage
of the Duke-Imamoglu-Ibukiyama-Ikeda lift $F$.
The $m$-th Fourier-Jacobi coefficient $e_{k+n,m}^{(2n-1)}$ of Siegel-Eisenstein series
satisfies also the identity (\ref{id:phi_m_V}).
Thus, because of Theorem~\ref{thm:eins} and of Lemma~\ref{lem:Adash},
we obtain
\begin{eqnarray*}
  \begin{pmatrix}
     \Phi_{0,0}(p^{k-\frac12}) & \cdots & \Phi_{0,2n-1}(p^{k-\frac12})\\
     \Phi_{1,0}(p^{k-\frac12}) & \cdots & \Phi_{1,2n-1}(p^{k-\frac12})\\
     \Phi_{2,0}(p^{k-\frac12}) & \cdots & \Phi_{2,2n-1}(p^{k-\frac12})
   \end{pmatrix}
  &=&
  \begin{pmatrix}
    0 & 1 \\
    p^{-n-\frac12} & p^{-n-\frac12}(-1+p\, \delta_{p|m})\\
    0 & p^{-2n+1}
  \end{pmatrix}
  A'_{2,2n}(p^{-(k-\frac12)}).
\end{eqnarray*}
Furthermore, this identity holds for infinitely many $k$.
Hence we obtain
\begin{eqnarray}
  \label{id:Phi_X_p} 
  &&
  \\
  \nonumber
  \begin{pmatrix}
     \Phi_{0,0}(X_p) & \cdots & \Phi_{0,2n-1}(X_p)\\
     \Phi_{1,0}(X_p) & \cdots & \Phi_{1,2n-1}(X_p)\\
     \Phi_{2,0}(X_p) & \cdots & \Phi_{2,2n-1}(X_p)
   \end{pmatrix}
  &=&
  \begin{pmatrix}
    0 & 1 \\
    p^{-n-\frac12} & p^{-n-\frac12}(-1+p\, \delta_{p|m})\\
    0 & p^{-2n+1}
  \end{pmatrix}
  A'_{2,2n}(X_p^{-1}).
\end{eqnarray}
In particular, we get $A'_{2,2n}(X_p) = A'_{2,2n}(X_p^{-1})$.
Due to the identities
(\ref{id:phi_m_V}) and (\ref{id:Phi_X_p}),
we thus obtain Theorem~\ref{thm:phi_maass}.


%

\subsection{Proof of Theorem~\ref{thm:drei}}


From the identity~(\ref{eq:W_T_V}) in Section~\ref{ss:index_shift}
and from Theorem~\ref{thm:phi_maass} we obtain
 \begin{eqnarray*}
   &&
    \phi_m(\tau,0)|(T_{0,2n-1}(p^2),...,T_{2n-1,0}(p^2))\\
   &=&
     p^{2nk+n-1}
     \left(
       \phi_{\frac{m}{p^2}}(\tau,0), 
       \phi_{m}(\tau,0),
       \phi_{mp^2}(\tau,0)
     \right)
     \begin{pmatrix}
       0              & 1 \\
       p^{-k-n} & p^{-k-n}(-1+p\, \delta_{p|m})  \\
       0              & p^{-2k-2n+2}
     \end{pmatrix}
     A'_{2,2n}(\alpha_p).
 \end{eqnarray*}
Due to the identity
  $\mathcal{F}\left(\begin{pmatrix}\tau&0\\0&\omega \end{pmatrix}\right)
    = \displaystyle{\sum_{m>0} \phi_m(\tau,0)e(m\omega)}$, we have
 \begin{eqnarray*}
  &&
    \sum_{m>0}
       \left\{
           \phi_{\frac{m}{p^2}}(\tau,0) + p^{-k-n}(-1+p\, \delta_{p|m}) \phi_{m}(\tau,0) + p^{-2k-2n+2} \phi_{mp^2}(\tau,0)
       \right\}
       e(m\omega)\\
  &=&
    p^{-2k-2n+2}\,
    \left. F\!\left(\begin{pmatrix} \tau&0\\0&\omega \end{pmatrix}\right) \right|_{\omega} T_{1,0}(p^2),
 \end{eqnarray*}
 where in the RHS
 we regard that the Hecke operator $T_{1,0}(p^2)$ acts on $F\left(\begin{pmatrix} \tau & 0 \\ 0 & \omega \end{pmatrix} \right)$
 as a function of $\omega \in \H_1$ for a fixed $\tau \in \H_1$.
 Therefore
 \begin{eqnarray*}
   &&
    \sum_{m>0} \phi_m(\tau,0)|(T_{0,2n-1}(p^2),...,T_{2n-1,0}(p^2)) e(m\omega) \\
   &=&
     p^{2nk+n-1}
     \left(
       p^{-k-n} F\left(\begin{pmatrix} \tau & 0 \\ 0 & \omega \end{pmatrix}\right),
       p^{-2k-2n+2} \left. F\left(\begin{pmatrix} \tau & 0 \\ 0 & \omega \end{pmatrix}\right) \right|_{\omega}T_{1,0}(p^2)
     \right)
     A'_{2,2n}(\alpha_p).
 \end{eqnarray*}
We denote by $\left<h_1(\omega),h_2(\omega)\right>_{\omega}$ the Petersson inner product of two elliptic modular forms $h_1$, $h_2$.
The symbol $\lambda_g(p^2)$ denotes the eigenvalue of $g$ for $T_{1,0}(p^2)$.

Because $\left<F\left(\begin{pmatrix} \tau&0\\0& \omega \end{pmatrix}\right),\, g(\omega)\right>_{\omega} = \mathcal{F}_{f,g}(\tau)$
and because
\begin{eqnarray*}
   \left<\left. F\left(\begin{pmatrix} \tau&0\\0& \omega \end{pmatrix}\right)\right|_{\omega} T_{1,0}(p^2),\, g(\omega)\right>_{\omega}
 &=&
   \lambda_{g}(p^2) \mathcal{F}_{f,g}(\tau),
\end{eqnarray*}
we obtain
 \begin{eqnarray*}
  &&
    \mathcal{F}_{f,g}|(T_{0,2n-1}(p^2),...,T_{2n-1,0}(p^2))\\
  &=&
    \left<\sum_{m>0} \phi_m(\tau,0)|(T_{0,2n-1}(p^2),...,T_{2n-1,0}(p^2)) e(m\omega),\ g(\omega)\right>_{\omega} \\
  &=&
    p^{2nk+n-1}
     \left(
       p^{-k-n} \mathcal{F}_{f,g},\
       p^{-2k-2n+2} \lambda_{g}(p^2) \mathcal{F}_{f,g}
     \right)
     A'_{2,2n}(\alpha_p).
 \end{eqnarray*}
Therefore we proved Theorem~\ref{thm:drei}.

\subsection{Proof of Corollary~\ref{coro:vier}}

%
%
%

Let $\{\mu_0,\mu_1,...,\mu_{2n-1} \}$ be the Satake parameter of $\mathcal{F}_{f,g}$ at a prime $p$.
We recall
 \begin{eqnarray*}
     A'_{2,2n}(X_p)
   &=&
     B'_{2,2n}(p^{\frac32-n} X_p, p^{\frac52-n} X_p ,..., p^{-\frac32+n} X_p),
 \end{eqnarray*}
where the matrices $A'_{2,2n}$ and $B'_{2,2n}$ are defined in Section~\ref{ss:Satake_Siegel_Phi}.
Because of the construction of $A'_{2,2n}(X_p)$,
the matrix $A'_{2,2n}(\alpha_p)$ determines a Satake parameter $\{\mu_2,...,\mu_{2n-1}\}$
up to the action of the Weyl group $W_{2n-1}$.
Hence we can take
\begin{eqnarray*}
    \{\mu_2,...,\mu_{2n-1}\}
  &=&
    \{p^{\frac32-n}\alpha_p,p^{\frac52-n}\alpha_p, ..., p^{-\frac32+n}\alpha_p \}.
\end{eqnarray*}

Now, from Section~\ref{ss:Satake_Siegel_Phi} and Section~\ref{ss:standard-L}, we recall
\begin{eqnarray*}
  &&
   (\varphi(T_{0,2n-1}(p^2)),\varphi(T_{1,2n-2}(p^2)),...,\varphi(T_{2n-1,0}(p^2)))\\
  &=&
   \left(\prod_{i=2}^{2n-1} X_i\right)
  (\varphi(T_{0,1}(p^2)),\varphi(T_{1,0}(p^2)))
   B'_{2,2n}(X_2,...,X_{2n-1})
\end{eqnarray*}
and
 $
    \mu_0^2 \mu_1 \cdots \mu_{2n-1}
  =
    p^{(2n-1)k}
 $,
where $\varphi$ is the Satake isomorphism denoted in Section~\ref{ss:Satake_Siegel_Phi},
and where $T_{l,2n-l}(p^2)$ $(l=0,...,2n)$ is the Hecke operator denoted
in~\ref{ss:index_shift}.
Furthermore, from a straightforward calculation we have
\begin{eqnarray*}
  \varphi(T_{0,1}(p^2)) &=& p^{-1} X_0^2 X_1,\\
  \varphi(T_{1,0}(p^2)) &=& p^{-1} X_0^2 X_1 (p X_1^{-1} + (p-1) + pX_1).
\end{eqnarray*}
From Theorem~\ref{thm:drei} and the above relations,
we have
\begin{eqnarray*}
   &&
    p^{2nk+n-1}
     \left(
       p^{-k-n},\
       p^{-2k-2n+2} \lambda_{g}(p^2)
     \right)
     A'_{2,2n}(\alpha_p)\\
   &=&
     \left(\prod_{i=2}^{2n-1} \mu_i\right)
     (p^{-1} \mu_0^2 \mu_1,\ p^{-1} \mu_0^2 \mu_1 (p \mu_1^{-1} + (p-1) + p \mu_1))
      B'_{2,2n}(\mu_2,...,\mu_{2n-1}).
\end{eqnarray*}
Hence, from the fact
that the rank of the matrix $A'_{2,2n}(\alpha_p)$ is two, we obtain
\begin{eqnarray*}
   p \mu_1^{-1} + (p-1) + p \mu_1
  &=&
   p^{-k-n+2} \lambda_g(p^2).
\end{eqnarray*}
On the other hand, we have $\lambda_g(p^2) = p^{k+n-2}(p \beta_p^2 + (p-1) + p \beta_p^{-2})$.
Thus we can take $\mu_1 = \beta_p^2$.
Hence we obtain
\begin{eqnarray*}
    \left\{ \mu_1,\mu_2,\mu_3...,\mu_{2n-1}\right\}
  &=&
    \left\{\beta_p^2,p^{\frac32-n}\alpha_p,p^{\frac52-n}\alpha_p,...,p^{-\frac32+n}\alpha_p\right\}
\end{eqnarray*}
up to the action of the Weyl group $W_{2n-1}$.

%
%

The standard $L$-function of $\mathcal{F}_{f,g}$ is
\begin{eqnarray*}
   L(s,\mathcal{F}_{f,g},\mbox{st})
 &=&
   \prod_p\left\{(1-p^{-s})\prod_{i=1}^{2n-1}\left\{(1-\mu_i p^{-s})(1-\mu_i^{-1} p^{-s}) \right\}\right\}^{-1}\\
 &=&
   \prod_p\Biggl\{ (1-p^{-s})(1-\beta_p^2 p^{-s})(1-\beta_p^{-2} p^{-s}) \\
 && \left. \times
   \prod_{i=1}^{2n-2}\left\{(1-\alpha_p p^{i+\frac12-n-s})
                            (1-\alpha_p^{-1} p^{-i-\frac12+n-s}) \right\} \right\}^{-1}\\
 &=&
   L(s,g,\mbox{Ad}) \prod_p
   \prod_{i=1}^{2n-2}\left\{(1-\alpha_p p^{i+\frac12-n-s})
                            (1-\alpha_p^{-1} p^{i+\frac12-n-s}) \right\}^{-1}.
\end{eqnarray*}
Because $L(s,f) = \displaystyle{\prod_{p}\left\{(1-\alpha_p p^{k-\frac12-s})(1-\alpha_p^{-1} p^{k-\frac12-s})\right\}^{-1}}$,
we obtain Corollary~\ref{coro:vier}.



\vspace{1cm}

\noindent
Department of Mathematics, Joetsu University of Education,\\
1 Yamayashikimachi, Joetsu, Niigata 943-8512, JAPAN\\
e-mail hayasida@juen.ac.jp

\end{document}